\documentclass[12pt,reqno]{amsart}
\usepackage {amssymb}
\usepackage {amsmath}
\usepackage {bbm}
\usepackage{amsthm}
\usepackage{mathtools}
\usepackage{graphicx}
\usepackage {amscd}
\usepackage {epic}
\usepackage[alphabetic]{amsrefs}
\usepackage[colorlinks, citecolor=blue]{hyperref}
\usepackage{enumerate}
\usepackage{tikz}
\usepackage{tikz-cd}
\usepackage{verbatim,color,geometry}
\usepackage[all]{xy}
\newcommand{\Aut}[0]{\operatorname{Aut}}

\newcommand{\PP}{{\mathbb{P} }}
\newcommand{\Z}{{\mathbb{Z} }}

\newcommand{\cR}{{\mathcal{R}}}
\newcommand{\cB}{{\mathcal{B}}}
\newcommand{\cX}{{\mathcal{X}}}

\newcommand{\cZ}{{\mathcal{Z}}}

\newcommand{\mult}{{\rm mult}}
\newcommand{\Supp}{{\rm Supp}}

\newcommand{\ord}{{\rm ord}}
\newcommand{\gr}{{\rm gr}}

\newtheorem{thm}{Theorem}[section]

\newtheorem{lem}[thm]{Lemma}
\newtheorem{cor}[thm]{Corollary}

\newtheorem{prop}[thm]{Proposition}

\newtheorem{conj}[thm]{Conjecture}

\theoremstyle{definition}
\newtheorem{defn}[thm]{Definition}

\newtheorem{rem}[thm]{Remark}

\newtheorem{defn-thm}[thm]{Definition--Theorem}  
\newtheorem{defn-prop}[thm]{Definition--Proposition}  
\newtheorem{defn-lem}[thm]{Definition--Lemma}  

\theoremstyle{remark}

\input{diagrams.sty}

\newcommand{\Fut}{{\rm Fut}}
\newcommand{\Proj}{{\rm Proj}}
\newcommand{\Val}{{\rm Val}}
\newcommand{\vol}{{\rm vol}}
\newcommand{\hvol}{\widehat{\rm vol}}
\newcommand{\cD}{{\mathcal{D}}}

\newcommand{\cO}{{\mathcal{O}}}

\newcommand{\fa}{{\mathfrak{a}}}
\newcommand{\lct}{{\rm lct}}
\newcommand{\im}{{\rm im}}
\newcommand{\bQ}{{\mathbb{Q}}}
\newcommand{\Q}{{\mathbb{Q}}}
\newcommand{\N}{{\mathbb{N}}}
\newcommand{\fb}{{\mathfrak{b}}}

\newcommand{\Spec}{{\rm Spec}}

\renewcommand{\d}{\delta}
\newcommand{\Diff}{{ \rm Diff}}
\newcommand{\e}{\varepsilon}
\newcommand{\Exc}{{ \rm Exc}}

\newcommand{\la}{\lambda}
\newcommand{\cI}{\mathcal{I}}

\newcommand{\coeff}{{ \rm coeff}}
\newcommand{\fc}{{\mathfrak{c}}}
\newcommand{\cF}{\mathcal{F}}

\newcommand{\init}{{\rm in}}

\newcommand{\R}{{\mathbb{R}}}

\begin{document}

\title{Uniqueness of K-polystable degenerations of
Fano varieties}
\date{\today}

\author{Harold Blum}
\address{University of Utah, Salt Lake City, UT 48112, USA}
\email{blum@math.utah.edu}

\author {Chenyang Xu}

\address   {Beijing International Center for Mathematical Research,
       Beijing 100871, China}
\email     {cyxu@math.pku.edu.cn}
\address   {Current Address:
 MIT, Cambridge, MA 02139, USA}
\email     {cyxu@math.mit.edu}

\begin{abstract}{
We prove that K-polystable  degenerations of $\mathbb{Q}$-Fano varieties are unique. 
Furthermore, we show that the moduli stack of K-stable $\mathbb{Q}$-Fano varieties is separated. 
Together with recently proven boundedness and openness statements, the latter result yields a separated Deligne-Mumford stack parametrizing all uniformly K-stable  $\Q$-Fano varieties of fixed dimension and volume.  The result also implies that the automorphism group of a K-stable $\Q$-Fano variety is finite.}
\end{abstract}

\maketitle{}
\setcounter{tocdepth}{1}
\tableofcontents

\section{Introduction}\label{s-intro}
\subsection{Moduli spaces of Fano varieties}\label{ss-moduli}
To give a general framework for intrinsically constructing  moduli spaces of Fano varieties is a challenging question in algebraic geometry, especially if one wants to find a compactification.  Unlike the KSBA construction in the canonically polarized case,  the Minimal Model Program often provides more than one limit for a family of Fano varieties over a punctured curve.  Thus, it is unclear how to find a  theory that picks the right limit. In examples, people have obtained a lot of working experience on how to identify the simplest limit. On the negative side, examples such as \cite[Section 2.2]{PP10}, which gives a family that  isotrivially degenerates a homogeneous space to a different quasi-homogeneous space (with non-reductive automorphism group), suggest that we should not consider all smooth Fano varieties.

So when the definition of K-stability from complex geometry \cite{Tia97} and its algebraic formulation  \cite{Don02}, which were introduced to characterize when a Fano variety admits a K\"ahler-Einstein metric, first appeared in front of algebraic geometers, it seemed bold to expect such a notion would be a key ingredient in constructing moduli spaces of Fano varieties.
However, as the theory has developed,  more and more evidence makes such an expectation believable.

We now expect that the moduli functor $\mathcal{M}^{\rm Kss}_{n,V}$ of $n$-dimensional K-semistable $\mathbb{Q}$-Fano varieties of volume $V$, which sends $S\in {\sf Sch}_k$ to
 \[
\mathcal{M}^{\rm Kss}_{n,V}(S) = 
 \left\{
  \begin{tabular}{c}
\mbox{Flat proper morphisms $X\to S$, whose geometric fibers are }\\
\mbox{$n$-dimensional K-semistable $\Q$-Fano varieties with  }\\
\mbox{volume $V$, satisfying Koll\'ar's condition}
  \end{tabular}
\right\}
\]
is represented by an Artin stack $\mathcal{M}^{\rm Kss}_{n,V}$ of finite type and admits a projective good moduli space $\mathcal{M}^{\rm Kss}_{n,V}\to M^{\rm Kps}_{n,V}$ (in the sense of \cite{Alp13}), whose closed points are in bijection with $n$-dimensional K-polystable $\Q$-Fano varieties of volume $V$. Here,  Koll\'ar's condition means that the reflexive power $\omega^{[m]}_{X/S}$ is flat over $S$ and commutes with arbitrary base change for each $m \in \Z$ (see \cite[24]{Kol08}).

While smooth K\"ahler-Einstein  Fano manifolds with finite automorphism group are asymptotically Chow stable  \cite{Don01}, examples in \cite{OSY12, LLSW17} show that the GIT approach likely fails to treat those with infinite automorphism groups or singularities. (See \cite{WX14} for examples where asymptotic Chow stability fails to construct compact moduli spaces in the KSBA setting.) Therefore, we need to take a more abstract approach to constructing $M^{\rm Kps}_{n,V}$.

The construction of $M^{\rm Kps}_{n,V}$ reduces to proving a number of concrete statements about families of $\Q$-Fano varieties. We list the main ones: 
\begin{enumerate}[(I)]
\item \textsc{Boundedness}: There is a positive integer $N=N (n,V)$ such that if $X\in \mathcal{M}^{\rm Kss}_{n,V} (k)$, then $-NK_{X}$  is a very ample Cartier divisor. This is settled in \cite{Jia17} using results in \cite{Bir16}.
\item \textsc{Zariski openness}: If $X \to S$ is a family of $\mathbb{Q}$-Fano varieties, then the locus where the fiber is K-semistable is a Zariski open set. 
\end{enumerate}

\medskip

Together, (I) and (II) show that $\mathcal{M}^{\rm Kss}_{n,V}$ is an Artin stack of finite type and is a global quotient. 
The following  statements  are needed to show $\mathcal{M}^{\rm Kss}_{n,V}$  admits a projective good moduli space.

\medskip

\begin{enumerate}[(I)]
\setcounter{enumi}{2}
\item \textsc{Good quotient}: The stack $\mathcal{M}^{\rm Kss}_{n,V}$ admits a good  moduli space. To prove this, it suffices to show:
\begin{enumerate}
\item[(III.a)] \textsc{Reductive automorphism group}: If $X$ is a K-polystable $\mathbb{Q}$-Fano variety $X$, then ${\rm Aut}(X)$ is reductive. 
\item[(III.b)] \textsc{Gluing of local quotients}: 
Near each K-polystable $\mathbb{Q}$-Fano variety $X \in \mathcal{M}^{\rm Kss}_{n,V}(k)$, 
there exists a local atlas around $[X]$ given by an $\Aut(X)$ slice. Furthermore, a point in the slice is GIT (poly/semi)stable with respect to $\Aut(X)$ if and only if the corresponding $\Q$-Fano variety is K-(poly/semi)stable.  To complete this step, it remains to verify  that the local GIT quotient spaces glue together to give the good quotient $M^{\rm Kps}_{n,V}$ (e.g. the hypotheses of  \cite[Theorem 1.2]{AFS17} are satisfied).
\end{enumerate}

\item \textsc{Separatedness}: 
Any two K-semistable degenerations of a family of K-semistable $\Q$-Fano varieties over a punctured curve $C^{\circ}=C\setminus 0$ lie in the same $S$-equivalence class, i.e. they degenerate to a common K-semistable $\Q$-Fano variety via special test configurations. 

\item \textsc{Properness}:  
Roughly speaking, any family of K-semistable Fano varieties over a punctured curve $C^{\circ}=C\setminus 0$ can be filled in over $0$ to a family of K-semistable $\Q$-Fano varieties over $C$. 

\item \textsc{Projecitivty}: A sufficiently divisible multiple of the CM-line bundle yields an ample line bundle on $M^{\rm Kps}_{n,V}$.
\end{enumerate}

We note that there are subtleties related to the requirement that objects in $\mathcal{M}^{\rm Kss}_{n,V}(S)$ satisfy Koll\'ar's condition. Luckily, such issues are of a local nature and have all been addressed in the construction of  the moduli space of KSBA stable varieties (see \cite{Kol08,Kol19}). 

Strong evidence for the above picture is that, aside from (VI) (the projectivity of $M^{\rm Kps}_{n,V}$), the problem is completely solved in \cite{LWX14} (see also \cite{SSY16, Oda15}) for $\mathbb{Q}$-Fano varieties with a $\Q$-Gorenstein smoothing and some progress on the projectivity was made in \cite{LWX15}. However, these results rely heavily on the deep analytic tools established in \cite{CDS,Tia15}. Therefore, a completely algebraic proof is highly desirable. Such a proof would likely allow us to drop the  \emph{smoothable} assumption. 

The main result in this paper  gives a complete solution to (IV).  In the smoothable case, this step is solved in \cite{LWX14, SSY16} using analytic tools. The argument in this document is purely algebraic. 

\subsection{Separatedness result}

The following statement is our main result.
\begin{thm}\label{t-main}
Let $\pi: (X,\Delta) \to C$ and $\pi': (X',\Delta')\to C$ be $\Q$-Gorenstein families of log Fano pairs over a smooth pointed curve $0 \in C$. Assume there exists an  isomorphism
\[
\phi: (X,\Delta) \times_C C^\circ  \to (X',\Delta') \times_C C^\circ
\] 
over $C^\circ\colon = C\setminus 0 $.
\begin{enumerate}[(1)]
\item K-semistable case:
If $(X_0,\Delta_0)$ and $(X_0',\Delta_0')$ are K-semistable, then they are S-equivalent.
 \item K-polystable case:
 If $(X_0,\Delta_0)$ and $(X_0',\Delta_0')$ are K-polystable, then they are isomorphic.
\item K-stable case:
 If $(X_0,\Delta_0)$ is K-stable and $(X_0',\Delta_0')$ is K-semistable, then $\phi$ extends to an isomorphism $(X,\Delta) \simeq (X',\Delta')$ over $C$. 
\end{enumerate}
\end{thm}

\begin{rem}

\begin{enumerate}
\item The K-polystable case of Theorem \ref{t-main} follows immediately from the K-semistable case and  Definitions \ref{d-kpoly} and \ref{d-sequiv}.

\item By \cite{LWX18}, the K-semistable case of Theorem \ref{t-main} can be strengthened to say $(X_0,\Delta_0)$ and $(X_0',\Delta_0')$ have a common K-polystable degeneration.

\item A special case of Theorem \ref{t-main} was proved in  \cite[1.4]{Oda12} with the additional assumption that $\alpha(X_0,\Delta_0)>\dim(X_0)/ (\dim(X_0)+1)$ (see also \cite[5.7]{Che09}).
\end{enumerate}
\end{rem}

Theorem \ref{t-main} implies the following special case of Step (III.a).

\begin{cor}\label{c-aut}
Let $(X,\Delta)$ be log Fano pair. If $(X,\Delta)$ is K-stable, then $\Aut(X,\Delta)$ is finite.
\end{cor}

\subsection{Moduli of uniformly K-stable Fano varieties}
We now specialize our study to the moduli of uniformly K-stable Fano varieties.
Consider the moduli functor $\mathcal{M}^{\rm uKs}_{n,V}$ that sends $S\in {\sf Sch}_k$ to 
\[
 \mathcal{M}^{\rm uKs}_{n,V}(S) =  \left\{
  \begin{tabular}{c}
\mbox{flat, proper morphisms  $X \to S$, whose geometric fibers  }\\
\mbox{ are $n$-dimensional uniformly K-stable $\Q$-Fano varieties }\\
\mbox{  of volume $V$, satisfying Koll\'ar's condition}
  \end{tabular}
\right\}.\]

\noindent Combining the following recent results:
\begin{enumerate}[(I$^{\rm u}$)]
\item \textsc{Boundedness}: Proved in  \cite{Jia17}, 
\item \textsc{Zariski openness}: Proved in  \cite{BL18}, and
\item \textsc{Separatedness (as a stack)}: Theorem \ref{t-main}.3,
\end{enumerate}
we obtain the following corollary. 
\begin{cor}\label{c-moduli}
The  functor $\mathcal{M}^{\rm uKs}_{n,V}$ is a separated Deligne-Mumford stack of finite type, which has a coarse moduli space $M^{\rm uKs}_{n,V}$ that is a separated algebraic space. 
\end{cor}
One still missing property is 

\smallskip
\begin{enumerate}[(IV$^{\rm u}$)]
\item \textsc{Quasi-projectivity}: $M^{\rm uKs}_{n,V}$ is quasi-projective. 
\end{enumerate}

\smallskip

\noindent{Significant progress on this problem was made in \cite{CP18}.

\medskip

\subsection{Summary of the paper}
The original definition of K-stability in \cite{Tia97, Don02} is defined in terms of the sign of  the generalized Futaki invariant on all test configurations or at least special test configurations (see \cite{LX14}).
Recently, there has been tremendous progress in reinterpreting K-stability in terms of invariants associated to valuations rather than test configurations.

 More specifically, in \cite{BHJ17}, the data of a test configuration was translated into the data of a filtration and it was shown that a nontrivial special test configuration yields a divisorial valuation.
 Then in a series of papers \cite{Fuj16a, Fuj16, Fuj18} of K. Fujita, all divisorial valuations were studied and an invariant $\beta$ was defined for each divisorial valuation.  
After \cite{Li17}, it became more natural to extend the setup to all valuations over the log Fano variety rather than only considering divisorial valuations (see also \cite{LX16, BJ17}). 
Moreover, a characterization  of K-stability notions in terms of the sign of $\beta$-invariant for divisorial valuations was proved in \cite{Li17, Fuj16} and lead to another characterization by the $\delta$-invariant in \cite{FO16, BJ17}.
These interpretations of K-stability using valuations have made it easier to apply techniques from birational geometry, 
especially the Minimal Model Program, to the study of K-stability.

In Section \ref{s-prelim}, we will have a short discussion on the above materials.
More precisely, we will provide information on the language of valuations and filtrations following \cite{BHJ17, Fuj17, Li17}
 and the invariants $\beta$ and $\delta$ associated to them following \cite{Fuj17, FO16, BJ17}. We also discuss the normalized volume function from \cite{Li18} and its relation with the K-stability of Fano varieties (see \cite{Li17, LX16}).

To proceed with our discussion, let us define the above invariants. 
Let $(X,\Delta)$ be a log Fano pair. 
Given a divisor $E$ over $X$ (i.e. $E\subset Y$ is a prime divisor, where $Y$ is a normal variety with a proper birational morphism $\pi:Y \to X$),  the $\beta$-invariant of $E$ is given by
\[
\beta_{X,\Delta}(E) : = A_{X,\Delta}(E)(-K_X-\Delta)^{n} - {\int^{\infty}_0 \vol(\pi^*(-K_X-\Delta)-tE)dt}
\]
where  $A_{X,\Delta}(E)$ is the log discrepancy of $E$. 
This invariant was defined in \cite{Fuj18} and the K-(semi)stability of $(X,\Delta)$ can be phrased in terms of the positivity of $\beta_{X,\Delta}(E)$ \cite{Fuj16, Li17}. 

Next, is the  $\d$-invariant of $(X,\Delta)$, which, as defined in \cite{FO16},  measures log canonical thresholds of a certain classes of anti-log canonical divisors of $(X,\Delta)$.
It is shown in \cite{BJ17} that
\begin{equation}\label{eq-deltaval}
\delta(X,\Delta)=\inf_{E} \frac{A_{X,\Delta}(E)(-K_X-\Delta)^{n}}{\int^{\infty}_0 \vol(\pi^*(-K_X-\Delta)-tE)dt}.
\end{equation}
Hence, we say that a divisor $E$ over $X$ computes $\d(X,\Delta)$ if it achieves the infimum in \eqref{eq-deltaval}.
The pair  $(X,\Delta)$ is uniformly K-stable (resp. K-semistable) if and only if $\delta(X,\Delta)>1$ (resp. $\delta(X,\Delta)\ge 1$)\cite{FO16,BJ17}.

\bigskip

In Section \ref{s-uniK}, before attacking Theorem \ref{t-main} in full generality, we consider the special case  in which $(X_0,\Delta_0)$ is uniformly K-stable and $(X'_0,\Delta'_0)$ is K-semistable. In this case, we provide a short proof of the separatedness result  by using properties of the $\d$-invariant to reduce the question to the well known separatedness statement for the moduli functor of klt log Calabi-Yau pairs (Proposition \ref{p-filling}). 
This argument is more straightforward than the general case and takes a slightly different approach. We hope this perspective can be applied in other cases.

\bigskip 
 
To prove Theorem \ref{t-main} in  full generality is more involved. 
We need to study the case  when the $\d$-invariants of the special fibers equal one. 
In general, analyzing the valuation computing $\d=1$ is quite subtle. 
For instance, the following statement has been conjectured by experts.
\begin{conj}\label{conj-special}
Let $(X,\Delta)$ be a log Fano pair. If $\delta(X,\Delta)\le 1$ then $\d(X,\Delta)$ is  computed by a divisor over $X$ and any such divisor is dreamy. 
\end{conj}

The special case of  Conjecture \ref{conj-special} when $\d(X,\Delta)=1$ implies that K-stability is equivalent to the apparently stronger notion of uniform K-stability. 
This is known for smooth Fano varieties by \cite{BBJ15}, but the proof relies on analytic tools, in particular the existence of K\"ahler-Einstein metrics. 

\smallskip

In Section \ref{s-delta1}, we will prove some special cases of Conjecture \ref{conj-special} which are needed in our proof of Theorem \ref{t-main}. The first result is that if $(X,\Delta)$ is a log Fano variety with $\d(X,\Delta)=1$, then any divisor computing $\d(X,\Delta)$ is necessarily dreamy and  induces a special test configuration of $(X,\Delta)$.
The proof relies on the MMP techniques developed in \cite{LWX18}, which  build upon work in \cite{Li17, LX16,LX17}.
Specifically, we consider the cone over our log Fano pair  and use the calculation in \cite{Li17, LX16} which shows that $\beta_{X,\Delta}(E)$ equals the derivative of the normalized volume function on the valuation space of the cone along the path given by the interval connecting the divisorial valuation associated to the pull back of $E$ and the canonical valuation. 
A careful study as in \cite[Theorem 3.2]{LWX18} shows that $E$ is indeed a dreamy divisor and induces a special test configuration.
In Section \ref{ss-intersection} and \ref{ss-qdivisor}, we also address the situations when the $\delta$-invariant can be calculated by an ideal or a $\mathbb{Q}$-divisor. These results may be of independent interest.

\medskip

Section \ref{s-kes} is the core of this paper and where we prove Theorem \ref{t-main}. 
The majority of the work in this section is to construct the S-equivalence stated in the theorem.  

\medskip

\noindent {\bf Step 1:}
 We first observe that a pair of two different degenerations will induce filtrations on each other's section rings. Furthermore, the associated  graded rings of the filtrations are isomorphic with a grading shift matching the calculation of $\beta$-invariant. 

Let us explain the above construction in more detail. Assume we have two $\Q$-Gorenstein families of log Fano pairs $\pi:(X,\Delta)\to C$ and $\pi':(X',\Delta') \to C$ over a smooth affine curve $C$ and 
an isomorphism
\[
\phi: (X,\Delta) \times_C C^\circ  \to (X',\Delta') \times_C C^\circ,
\] 
that does not extend to an isomorphism over $C^\circ = C\setminus 0 $. Fix $r$ so that $L:=-r(K_X+\Delta)$ and $L:=-r(K_{X'}+\Delta')$ are Cartier. 
We choose  a proper birational model over $X$ and $X'$ 
\begin{center}
\begin{tikzcd}[column sep=scriptsize, row sep=scriptsize]
  & \widehat{X}  \arrow[dr, "\psi'"] \arrow[dl,swap,"\psi"]& \\
  X  \arrow[rr, dashrightarrow, "\phi"] & & X'
   \end{tikzcd}
   \end{center}
and write $\widetilde{X}_0$ and $\widetilde{X}'_0$ for the birational transforms of $X_0$ and $X'_0$ on $\widehat{X}$. 
The divisor $\widetilde{X}'_0$ induces a filtration $\cF$ on the section ring of $(X_0,\Delta_0)$ defined by
\[
s \in \cF^p H^0(X_0, mL_0 ) \quad \text{ if and only if} \quad  \ord_{\widetilde{X}'_0}(\tilde{s}) \geq p 
\]
for some (non-unique) extension $\tilde{s} \in H^0(X,mL )$. Then we define
\[
\beta := (-K_{X_0}-\Delta_0)^n  A_{X,\Delta+ X_0}( \widetilde{X}'_0)   - \lim_{m \to \infty}  \int_0^\infty \frac{ \dim \cF^{mx} H^0(X_0,mL_0 )}{r^{n+1} m^n /n!} \, dx 
. \]
Similarly, we can define a filtration $\cF'$ of the section ring of $(X'_0,\Delta'_0)$ and the value $\beta'$ using the divisor $\widetilde{X}_0$. The construction here can be viewed as a relative version of the one  in \cite[Section 5]{BHJ17}, where they consider a test configuration and a trivial family. 

Next, we observe that there is an isomorphism of the associated graded rings of the filtrations
\begin{eqnarray}\label{e-isograde}
 \bigoplus_{m \in \N}\bigoplus_{p\in \Z}{\rm gr}^p_\cF H^0(X_0,mL_0) \overset{\varphi}{\longrightarrow} \bigoplus_{m \in \N}\bigoplus_{p\in \Z}{\rm gr}_{\cF'}^{mr(a+a')-p}H^0(X'_0, mL'_0),
\end{eqnarray}
where $a: = A_{X,\Delta+X_0}( \widetilde{X}'_0)$ and $a': = A_{X',\Delta'+X'_0}( \widetilde{X}_0)$.
Using this isomorphism, we deduce that $\beta+\beta'=0$.

Now if we assume $(X_0,\Delta_0)$ and $(X'_0,\Delta'_0)$ are K-semistable, then the $\beta$-invariant of any divisor over $X_0$ or $X'_0$ is non-negative \cite{Fuj16, Li17}. A similar result is extended to filtrations in \cite{BL18}.\footnote{This is also independently obtained by Chi Li and Xiaowei Wang in \cite{LW18}.} We can then conclude that $\beta = \beta'=0$. 

\medskip

\noindent {\bf Step 2:} At this point, we know  $X_0$ and $X'_0$ have a common degeneration. 
Indeed,  the Rees construction gives degenerations 
\[
X_0 \leadsto
 \cX_0 : = \Proj \bigg( \bigoplus_{ m \in \N}  \bigoplus_{p \in \Z} \gr^p_{\cF} H^0(X_0, mL_0) \bigg)\]
and
\[
X_0' \leadsto \cX'_0 : = \Proj  \bigg( \bigoplus_{ m \in \N} \bigoplus_{ p \in \Z }  \gr_{\cF'}^p H^0(X'_0, mL'_0) \bigg). 
\]
By \eqref{e-isograde}, the degenerations $\cX_0$ and $\cX'_0$ are isomorphic.

An immediate concern is that the above graded rings are not necessarily finitely generated. 
(Note that notions of K-stability have been investigated in the setting of non-finitely generated filtrations \cite{WN12, Sze15}.)
Since we aim to prove $(X_0,\Delta_0)$ and $(X'_0,\Delta'_0)$  have a common degeneration to a K-semistable log Fano pair, we must show that the filtrations $\cF$ and $\cF'$ are finitely generated and induce special test configurations with generalized Futaki invariant zero. By \cite[3.1]{LWX18}, this will imply that the degenerations are
K-semistable log Fano pairs. 

To proceed, we rely on the fact that our filtrations are induced by divisors over our families. 
More precisely, we use that $\beta=0$ to show that there exists an extraction  
$Y \to X$  of $\widetilde{X}'_0$ and the fiber $Y_0 = V \cup W$, 
where $V$ and $W$ are the birational transforms of $X_0$ and $X'_0$. Now, we set $E =W\vert_{V}$ and observe that $E$ induces a filtration $\cF_E$ on the section ring of $(X_0,\Delta_0)$. 
We then show $F:=\Supp(E)$ is a prime divisor and $\beta_{X_0,\Delta_0}(F)=0$.
Using Theorem \ref{t-delta=1}, we see  $\cF_E$ is finitely generated and the corresponding degeneration of $(X_0,\Delta_0)$ is a special test configuration with
generalized Futaki invariant zero.  

Next, we seek to show that the filtrations $\cF$ and $\cF_E$ are equal. 
This statement is equivalent to the surjectivity of certain restriction maps and is non-trivial. 
To achieve the result,  we take the relative cone of $(X,\Delta)$ over $C$ and run an analysis similar to the proof of Theorem \ref{t-delta=1}. 
After completing this argument, we can conclude that the degenerations $(\cX_0,\cD_0)$ and $(\cX'_0,\cD'_0)$ are naturally K-semistable pairs. 

Finally, we need to show that the isomorphism $\cX_0 \simeq \cX'_0$
sends the degeneration of $\Delta_0$ to the degeneration of $\Delta'_0$, so that we get an isomorphism of pairs. 
 To verify this, we choose a divisor $B \subseteq \Supp(\Delta)$ and write $B'\subseteq \Supp(\Delta')$ for its birational transform. 
 Now, $B_0$ degenerates to a divisor on $\cX_0$ that corresponds to the initial ideal $\init(I_{B_0})$ in the associated graded ring. 
 Rather than showing  $\varphi (\init(I_{B_0})) =\init(I_{B'_0})$, we introduce auxiliary ideals $I$ and $I'$ such that the equality $\varphi(I) = I'$ is clear.
(The ideal $I$  is defined by restricting elements of the relative section ring  that vanish to certain orders along $B$ and $\widetilde{X}'_0$.)
Using the relative cone construction again,
we show that $I$ and $I'$ agree with $\init(I_{B_0})$ to $\init(I_{B'_0})$ at codimension one points.
We can then conclude that the desired isomorphism of boundaries holds.

\bigskip

\noindent{\bf Acknowledgement: } We are grateful to 
Roman Bezrukavnikov,  
Giulio Codogni,
Tommaso de Fernex,  
Christopher Hacon,
Mattias Jonsson, 
J\'anos Koll\'ar,
Chi Li, 
Yuchen Liu,
Mircea Musta\c{t}\u{a}, 
Yuji Odaka, 
Xiaowei Wang, and 
Jun Yu for helpful conversations and comments on previous drafts of this paper. 
A large part of the work on this paper was completed when CX was visiting Institut Henri Poincar\'e as part of the Poincar\'e Chair program.
 He thanks the institute for the wonderful  environment and Claire Voisin for her hospitality.  
 Finally, we are indebted to the anonymous referees whose comments improved this paper considerably.
 
HB is partially supported by NSF grant DMS-1803102.
CX is partially supported by the National Science Fund for Distinguished Young Scholars (NSFC 11425101) ``Algebraic Geometry".

\section{Preliminaries on valuations and K-stability}\label{s-prelim}

\subsection{Conventions} 
We work over an algebraically closed  field $\mathbbm{k}$ of characteristic 0.
We follow the terminologies in \cite{KM98, Kol13}.
A \emph{pair} $(X,\Delta)$ is composed of a normal variety $X$ and an effective $\Q$-divisor $\Delta$
on $X$ such that $K_X+\Delta$ is $\mathbb{Q}$-Cartier. 
See \cite[2.34]{KM98} for the definitions of \emph{klt}, \emph{plt}, and \emph{lc} pairs.
A pair $(X,\Delta)$ is \emph{log Fano} if $X$ is projective, $(X,\Delta)$ is klt, and $-K_X-\Delta$ is ample. 
A variety $X$ is {\it $\mathbb{Q}$-Fano} if $(X,0)$ is log Fano.

\begin{defn}
A \emph{$\Q$-Gorenstein family of log Fano pairs} $\pi:(X,\Delta) \to C$ over a smooth curve $C$ is composed of a flat proper morphism $\pi : X\to C$ and an effective $\Q$-divisor $\Delta$, not containing any fiber of $\pi$, satisfying: 
\begin{enumerate}
\item $\pi$ has normal, connected fibers (hence, $X$ is normal as well)
\item $-K_{X} -\Delta$ is $\Q$-Cartier and $\pi$-ample, and
\item $(X_t,\Delta_t)$ is klt for all $t\in C$. 
\end{enumerate}
\end{defn}

\subsection{Valuations}
Let $X$ be a variety. A valuation on $X$ will mean a valuation $v: K(X)^\times \to \R$  that is trivial on $\mathbbm{k}$ and has center on $X$. Recall, $v$ has center on $X$ if there exists a point $\xi\in X$ such that $v\geq 0$ on $\cO_{X,\xi}$ and $>0$ on the maximal ideal. Since $X$ is assumed to be separated, such a point $\xi$ is unique, and we say $v$ has \emph{center} $c_X(v): = \xi$.  
See \cite[3.1]{JM12} for the definition of quasimonomial valuations.

Following  \cite{JM12,BdFFU15}, we write $\Val_{X}$ for the set of valuations on $X$ and $\Val_X^*$ for the set of non-trivial ones. To any valuation $v\in \Val_{X}$ and $p\in \N$, there is an associated \emph{valution ideal} $\fa_{p}(v)$ .
For an affine open subset $U\subseteq X$, $\fa_p(v)(U) = \{ f\in \cO_X(U) \, \vert \, v(f) \geq p \}$ if $c_X(v) \in U$ and   $\fa_p(v)(U) =  \cO_X(U)$ otherwise.
 
For an ideal $\fa \subseteq \cO_X$ and $v\in \Val_X$, we set
\[
v(\fa) : = \min \{ v(f) \, \vert\, f\in \fa \cdot \cO_{X,c_X(v)} \} \in [0, +\infty].\]
We can define $v(s)$ when $\mathcal{L}$ is a line bundle on $X$ and $s\in H^0(X,\mathcal{L})$. After trivializing $\mathcal{L}$ at $c_X(v)$, we set $v(s)=v(f)$, where $f$ is the local function corresponding to $s$ under this trivialization (this is independent of choice of trivialization).

\subsubsection{Divisors over $X$}
Let $X$ be a variety and $\pi:Y \to X$ be a proper birational morphism, with $Y$ normal. 
A prime divisor $E \subset Y$ defines a valuation $\ord_E: K(X)^\times \to \Z$ given by order of vanishing at $E$. Note that $c_X(\ord_E)$ is the generic point of $\pi(E)$ and, assuming $X$ is normal, $\fa_p(v) = \pi_*\cO_Y(-pE)$.  

We identify two such prime divisors on  $Y_1$ and $Y_2$  as above if one is the birational transform of the other. Equivalently, they induce the same  valuation of $K(X)$. A \emph{divisor over} $X$ is  an equivalence class given by this relation.

\subsubsection{Log discrepancies}
Let $(X,\Delta)$ be a  pair. We write 
$$A_{X,\Delta}\colon \Val_X^* \to \R \cup \{ +\infty \} $$ 
for  the log discrepancy function with respect to  $(X,\Delta)$ as in \cite{JM12, BdFFU15}
(see \cite{Blu18b} for the case when $\Delta\neq 0 $).

When $\pi:Y \to X$ is a proper birational morphism with $Y$ normal and $E\subset Y$ a prime divisor, 
\[
A_{X,\Delta}( \ord_E) =  1+ \coeff_{E}\left( K_{Y}- \pi^*(K_X+\Delta) \right). 
\] We will often write $A_{X,\Delta}(E)$ for the above value.

The function $A_{X,\Delta}$ is homogenous of degree $1$, i.e. $A_{X,\Delta}(\lambda v ) = \lambda \cdot  A_{X,\Delta}(v)$ for $\lambda\in \R_{>0}$ and $v\in \Val_{X}$.
 A pair $(X,\Delta)$ is klt (resp., lc) if and only if $A_{X,\Delta}(v)>0$ (resp., $\geq0$) for all $v\in \Val_X^*$.

\subsubsection{Graded sequences}
A \emph{graded sequence of ideals}  $\fa_\bullet = (\fa_p)_{p\in \Z_{>0}}$ on a variety $X$ is a sequence of ideals on $X$ satisfying $\fa_p\cdot \fa_q\subseteq \fa_{p+q}$ for all $p,q\in \Z_{>0}$. By convention, $\fa_0 = \cO_X$.
We set $M(\fa_\bullet) : = \{ p \in \Z_{>0} \, \vert \, \fa_p \neq  (0) \}$ and  always assume $M(\fa_\bullet)$ is nonempty. If $v\in \Val_X^*$, then $\fa_\bullet(v)$ is a graded sequence of ideals.

Let $\fa_\bullet$ be a graded sequence of ideals on $X$ and $v\in \Val_X$. It follows from Fekete's Lemma  that the limit
\[
v(\fa_\bullet) :=  \lim_{M(\fa_\bullet) \ni m \to \infty}  \frac{v(\fa_m)}{m}
\]
exists and equals  $\inf_{m \geq 1} \frac{v(\fa_m)}{m}$; see \cite[\S2.1]{JM12}.

Let $x\in X$ be a closed point. If $\fa_\bullet$ is a graded sequence of ideals on $X$ and each ideal $\fa_p$ is $\mathfrak{m}_x$-primary, we set 
\[
\mult(\fa_\bullet) := \lim_{p \to \infty} \frac{\dim_{\mathbbm{k}} (\cO_X/\fa_p)}{p^n/n!}.\]
If $v\in \Val_X$ has center $\{x\}$, then $\fa_p(v)$ is $\mathfrak{m}_x$-primary for each $p>0$. In this case, we call $\vol(v): =\mult(\fa_\bullet(v))$ the \emph{volume} of $v$. 

\subsubsection{Log canonical thresholds}
Let $(X,\Delta)$ be an lc pair. Given a nonzero ideal $\fa \subseteq \cO_{X}$, the log canonical threshold of $\fa$ is given by  
\[
\lct(X,\Delta;\fa):= \sup \{c \in \Q_{\geq 0} \, \vert\, (X,\Delta+ \fa^c) \text{ is lc} \}.
\]
If $\fa_\bullet$ a graded sequence of ideals on $X$, the log canonical threshold of $\fa_\bullet$ is given by
\[
\lct(X,\Delta; \fa_\bullet) := \lim_{M(\fa_\bullet) \ni m \to \infty} m \cdot \lct(X,\Delta;\fa_m).\]
Fekete's Lemma implies that the above limit exists an equals
$  \sup_{m} m \cdot \lct(X,\Delta;\fa_m)$ \cite[2.5]{JM12}.

It is straightforward to show 
$
\lct(X,\Delta; \fa_\bullet) \leq  \frac{A_{X,\Delta}(v)}{v(\fa_\bullet)},
$
for $v \in \Val_{X}^*$ satisfying $0 \neq A_{X,\Delta}(v) < +\infty$. 
Hence, if $v\in \Val_{X}^*$ satisfies $A_{X,\Delta}(v) \neq 0$, then
\begin{equation}\label{e-lctA}
\lct(X,\Delta;\fa_\bullet(v)) \leq A_{X,\Delta}(v), 
\end{equation}
since  $v(\fa_\bullet(v))=1$ \cite[3.4.9]{Blu18b}.

\subsubsection{Extractions}

Let $E$ be a divisor over a normal variety $X$. 
We say that  $\mu:X_E \to X$ is an \emph{extraction} of $E$ if 
$\mu$ is a proper birational morphism with $X_E$ is normal,
 $E$ arises as a prime divisor  $E \subset X_E$, and $-E$ is $\mu$-ample.

Note that if $\mu:X_E \to X$ is an extraction of $E$, then $E \supseteq  \Exc(\mu)$
and equality holds if ${\rm codim}_{X}(  c_X(\ord_E) ) \geq 2$. 
Indeed,  Lemma \ref{l-antiample} implies that if $p\in \Z_{>0}$ is sufficiently  divisible, then $\mu$ is the blowup along $\fa_{p}(\ord_{E})$ 
and $\fa_{p}(\ord_E) \cdot \cO_Y = \cO_{Y}(-pE)$.
\medskip

The following technical statement gives a criterion for when an exceptional divisor may be extracted. The criterion will be used repeatedly in Section \ref{s-kes}. 

\begin{prop}\label{p-extract}
Let $(X,\Delta)$ be a klt pair or a plt pair such that $\lfloor\Delta\rfloor=S$ is a non-zero $\mathbb{Q}$-Cartier divisor. 
If $E$ is a divisor over $X$ satisfying 
$$ a:= A_{X,\Delta}(E) - \lct(X,\Delta; \fa_\bullet(\ord_E)) <1,$$ then 
there exists an extraction $\mu:X_E \to X$ of $E$ and $(X_E,\mu^{-1}_{*}(\Delta) + (1-a) E )$ is lc.
\end{prop}

The proposition is a consequence of \cite{BCHM10} 
and properties of the log canonical threshold of a graded sequence of ideals.

\begin{proof}
See the argument in \cite[1.5]{Blu17} for the case when $(X,\Delta)$ is klt. 
 If $(X,\Delta)$ is plt, observe that 
$(X,\Delta_{\varepsilon}:=\Delta-\varepsilon S)$ is klt for $0<\e <1$. 
If we set
\[
a_{\varepsilon}: = A_{X,\Delta_\varepsilon}( \ord_E) -  \lct(X,\Delta_\varepsilon; \fa_\bullet(\ord_E)),
\] 
then $\displaystyle \lim_{\varepsilon \to 0} a_\varepsilon = a$ and
we may reduce to the klt case.
\end{proof}

\subsection{Filtrations}
Let $(X,\Delta)$ be a $n$-dimensional log Fano pair. 
Fix a positive integer $r$ such that $L:=-r(K_X+\Delta)$ is Cartier and write
\[
R =R(X,L) =   \bigoplus_{m \in \N} R_{m}  = \bigoplus_{m \in \N}  H^0(X, \cO_X(mL))\]
for the section ring of $L$.  Set 
$M(L): = \{ m \in \N \, \vert \, H^0(X, \cO_X(mL)) \neq 0\}$.

\begin{defn}
A \emph{filtration} $\cF$ of $R$ we will mean the data of a family  of vector subspaces 
$\cF^\la R_m \subseteq R_m$  
for $m \in \N$ and $\la \in \R$ satisfying
\begin{itemize}
\item[(1)]
$\cF^\la R_m  \subseteq \cF^{\la'} R_m$ when $\la \geq \la'$;
\item[(2)] 
$\cF^\la R_m = \cap_{\la ' < \la} \cF^{\la'} R_m$;
\item[(3)]
$\cF^0 R_m =R_m$ and $\cF^\la R_m=0$ for $\la \gg0$.
\item[(4)]
$\cF^\la R_m \cdot \cF^\la R_{m'} \subseteq \cF^{\la +\la'} R_{m+m'}$.
\end{itemize}
A filtration $\cF$ of $R$ is a called an \emph{$\N$-filtration} if $\cF^\la R_m = \cF^{\lceil \la \rceil} R_m$ for all $m\in \N$ and $\la \in \R$. 
To give a $\N$-filtration $\cF$, it suffices to give a family of subspaces $\cF^p R_m \subseteq R_m$ for $m,p \in \N$ satisfying (1), (3), and (4). 

A filtration $\cF$ is  \emph{linearly bounded}  if there exists $C>0$ so that $\cF^{Cm} R_m=0$ for all $m \in \N$ and \emph{trivial} if $\cF^\la R_m =0$ for all $m \in \N$ and $\lambda >0$. 
\end{defn}

\subsubsection{Rees construction}
Let $\cF$ be an $\N$-filtration of $R$. 
The \emph{Rees algebra}  of $\cF$ is the $\mathbbm{k}[t]$-algebra
\[
{\rm Rees}(\cF): = \bigoplus_{m \in \N}  \bigoplus_{p \in \Z} (\cF^p R_m )  t^{-p} \subseteq R[t,t^{-1}].\]
The \emph{associated graded ring} of $\cF$ is
\[
\gr_{\cF} R : =  \bigoplus_{m \in \N} \bigoplus_{p \in \Z} \gr_{\cF}^p R_m,
\quad\quad \text{ where } \gr_{\cF}^p R_m = \frac{\cF^{p}R_m}{\cF^{p+1}R_m}
.\]
Note that 
\begin{equation}\label{e-Rees}
{ \rm Rees}(\cF) \otimes_{k[t]} k[t,t^{-1}] \simeq R[t,t^{-1}] \quad 
\text{ and } 
\quad\quad
\frac{{\rm Rees}(\cF)}{ t\, {\rm Rees}(\cF)}  \simeq \gr_{\cF}R .\end{equation}
Hence, ${\rm Rees}(\cF)$ is said to give a degeneration of $R$ to the associated graded ring of $\cF$.

An $\N$-filtration $\cF$ is \emph{finitely generated} if ${\rm Rees}(\cF)$ is a finitely generated $\mathbbm{k}[t]$-algebra.
In this case, we set  $\cX: = \Proj_{\mathbb{A}^1} \left({\rm Rees}(\cF)\right)$. 
By \eqref{e-Rees},
\[
\cX_{ \mathbb{A}^1\setminus 0 } \simeq X \times ( \mathbb{A}^1\setminus 0  ) 
\quad \text{ and }
\cX_0 \simeq \Proj ( \gr_{\cF} R ). 
\]
We write $\cD$ for the $\Q$-divisor that is the  closure  of $\Delta \times ( \mathbb{A}^1\setminus 0 )$ under the embedding of 
$X \times( \mathbb{A}^1\setminus 0  )$ in $\cX$.

The scheme $\cX$ can naturally be endowed with the structure of a \emph{test configuration} of $(X,\Delta)$.
The test configuration is called  \emph{special} if $(\cX,\cD) \to \mathbb{A}^1$ is a $\Q$-Gorenstein family of log Fano pairs.
See \cite[\S 3]{LX14} and \cite[\S 2]{BHJ17}  for information on test configurations and the generalized Futaki invariant.

With the above setup, consider a subscheme $Z\subset X$ 
and write $I_{Z} \subset R$ for the corresponding homogenous ideal.
The scheme theoretic closure of $Z\times (\mathbb{A}^1  \setminus 0 )$ in $\cX$, denoted by $\cZ$, is defined by the ideal  
\[
\bigoplus_{m \in \N} \bigoplus_{p \in \Z}  (\cF^p R_m \cap I_{Z})t^{-p} \subseteq {\rm Rees}(\cF).
\]
Indeed, the corresponding subscheme agrees with $Z\times (\mathbb{A}^1\setminus 0 )$  away from $0 $ and is torsion free over $0$.
The above description of $\cZ$ yields that its scheme theoretic fiber along 0 is given by the \emph{initial ideal}
\[
\init(I_Z) : = \bigoplus_{m \in \N} \bigoplus_{p \in \Z}   \im (\cF^p R_m \cap I_Z \to \gr_{\cF}^p R_m ) \subset \gr_{\cF}^p R_m
.\]

\subsubsection{Volume}
Given a filtration $\cF$ of $R$, we  set
\[
\vol( \cF R ^{(x)} ) := \limsup \limits_{m \to \infty}  \frac{ \dim (\cF^{xm} R_m ) }{ m^n /n!} 
\] for $x\in \R_{\geq0}$. 
Assuming $\cF$ is linearly bounded (which implies $\vol( \cF R ^{(x)}) =0$ for $x\gg0$), we set
\[
S(\cF) := \frac{1}{ r^{n+1}(-K_X-\Delta)^n }  \int_0^\infty \vol( \cF R^{(x)}) \, dx.
\footnote{Note that this differs from the definition of $S(\cF)$ in \cite{BJ17,BL18} by a factor of $1/r$. Since we are interested in the polarization $-K_X-\Delta$, not $L$,  such a convention is natural.}
\]
By \cite{BC11} (see also \cite[5.3]{BHJ17}), 
\begin{equation}\label{e-Riemsum}
 S(\cF) =
  \lim_{m \to \infty}  
  \left(
   \frac{1}{r \dim(R_m) } 
  \int_{0 }^\infty  \dim  (\cF^{mx} R_m ) \, dx \right)
   \end{equation}
   In particular, if  $\cF$ is an $\N$-filtration,  then 
$S(\cF) =   \lim_{m \to \infty} 
   \frac{  \sum_{p \geq 0 }  \left( p \dim \gr_{\cF}^p R_m  \right) }{ m r\dim R_m} $.

 \subsubsection{Base ideals}
Given a filtration $\cF$ of $R$,  set 
 $$\fb_{p,m}:= \im \left(  \cF^p R_m \otimes \cO_{X}(-mL) \to \cO_X\right).$$
 for $p,m\geq 0$. 
We set
$\fb_{p}(\cF): = \fb_{p,m}$ for $m\gg 1$.
The ideal $\fb_{p}(\cF)$ is well defined and $\fb_\bullet(\cF)$ is a graded sequence of ideals assuming $\cF$ is non-trivial \cite[3.17-3.18]{BJ17}. 

\subsubsection{Filtrations induced by valuations}
Given $v\in \Val_{X}^*$, we set 
\[
\cF_v^{\la } R_m = \{ s \in R_m \, \vert \, v(s) \geq \la \} 
\]
for each $\la\in \R$ and $m \in \N$.
Equivalently, $\cF_v^{\la } R_m = H^0(X, \cO_X(mL) \otimes \fa_\la(v))$.
Note that $\cF_v$ is a non-trivial filtration of $R$.

If $A_{X,\Delta}(v)< +\infty$ , then  $\cF_v$ is linearly bounded \cite[3.1]{BJ17}. In this case,  we set $S(v) :=S(\cF_v)$.

 \subsubsection{Filtrations induced by divisors}
If $E$ is a divisor over $X$, we write $\cF_E:=\cF_{\ord_E}$ and $S(E): = S(\cF_E)$.
Following \cite{Fuj16}, we say $E$ is \emph{dreamy} if $\cF_E$ is a finitely generated filtration of $R$. 

When $E$ arises as a prime divisor on a proper normal model $\mu:Y \to X$,
\[
\cF_E^{\la } R_m = H^0\big(Y, \cO_Y (  mr \mu^*(- K_X-\Delta)- \lceil \la \rceil E   ) \big)
.\] 
Therefore,
\[
S(E) = \frac{1}{ (-K_X-\Delta)^n } \int_0^\infty \vol( \mu^*(-K_X-\Delta)- x E) \, dx
.\]
\subsection{K-stability}
Based on the original analytic definition in \cite{Tia97}, an algebraic definition of K-(semi,poly)-stability was introduced in \cite{Don02}. Here, we will define these notations for log Fano pairs using valuations. 

\subsubsection{$\beta$-invariant}\label{ss-invariant}
Let $(X,\Delta)$ be an $n$-dimensional log Fano pair and $E$ a divisor over $X$. 
Following \cite{Fuj16}, 
$$\beta_{X,\Delta}(E):=(-K_X-\Delta)^n \left( A_{X,\Delta}(E) - S(E) \right).$$
More generally, if
 $v\in \Val_{X}$ with $A_{X,\Delta}(v) < +\infty$,
we set
$\beta_{X,\Delta}(v) := (-K_X-\Delta)^n \big( A_{X,\Delta}(v) -S(v) \big)$.

\begin{defn}\label{d-kstable}
A log Fano pair $(X,\Delta)$  is 
\begin{enumerate}
\item  \emph{K-semistable} if $\beta_{X,\Delta}(E)\ge 0$ for all divisors  $E$ over $X$;
\item \emph{K-stable} if $\beta_{X,\Delta}(E)> 0$ for all dreamy divisors  $E$ over $X$;
\item \emph{uniformly K-stable} if there exists an $\varepsilon>0$ such that 
$$\beta_{X,\Delta}(E)\ge \varepsilon A_{X,\Delta}(E)(-K_X-\Delta)^n$$ for all divisors $E$ over $X$.
 \end{enumerate}
 \end{defn}
The equivalence of the above definition with the original definitions was addressed in \cite{Fuj16,Fuj17,Li17} and the arguments rely on the special degeneration theory of \cite{LX14}. In Corollary \ref{c-alld}, we will show that the wordy \emph{dreamy} may be removed from Definition \ref{d-kstable}.2.

\begin{defn}\label{d-kpoly}
A log Fano pair $(X,\Delta)$ is \emph{K-polystable} if it is K-semistable and any special test configuration $(\cX, \cD)\to \mathbb{A}^1$ of $(X,\Delta)$ with  $(\cX_0,\cD_0)$ K-semistable satisfies $(\cX, \cD) \simeq (X,\Delta) \times \mathbb{A}^1$. 
\end{defn}

The equivalence of the above definition with the definition in \cite[6.2]{LX14}  relies on the following result: If $(X,\Delta)$ is a K-semistable log Fano pair and $(\cX,\cD)$ is a special test configuration of $(X,\Delta)$, then 
  $\Fut(\cX,\cD)=0$ if and only if $(\cX_0,\cD_0)$ is K-semistable \cite[3.1]{LWX18}.
   
 \begin{defn}\label{d-sequiv}
Two K-semistable log Fano pairs $(X,\Delta)$ and $(X', \Delta')$ are \emph{S}-equivalent 
if they degenerate to a common K-semistable log Fano pair via special test configurations. 
\end{defn}

By \cite[3.2]{LWX18}, S-equivalent log Fano pairs degenerate to a common K-polystable pair via special test configurations. 
Furthermore, the K-polystable pair is uniquely determined up to isomorphism.

\subsubsection{$\delta$-invariant}
We recall an interpretation of the above discussion using an invariant introduced in \cite{FO16}. 

Let $(X,\Delta)$ be a log Fano pair. Fix a positive integer $r$ so that $L:=-r(K_X+\Delta)$ is a Cartier divisor and $H^0(X,\cO_X(L)) \neq 0$. 
Given $m \in r \N$, we say $  D \sim_{\Q} -K_X-\Delta$ is  $m$-basis type  if there exists a basis 
$\{s_1,\ldots, s_{N_m}  \}$ of $H^0(X, \cO_{X}(-m(K_X+\Delta))$ such that \[
D = \frac{1}{m N_m} \big(  \left\{ s_1=0 \right\} + \cdots + \left\{ s_{N_m}=0 \right\} \big) .\]
We set 
$
\delta_{m}(X,\Delta): = \min \{ \lct(X,\Delta;D) \, \vert \, D \sim_{\Q} -K_X-\Delta \text{ is  $m$-basis type}\}$.
The \emph{$\delta$-invariant} (also known as the \emph{stability threshold}) of $(X,\Delta)$ is
\[\delta(X,\Delta) =  \limsup\limits_{m \to \infty } \delta_{mr} (X,\Delta), \]
and is independent of the choice of $r$ \cite[4.5]{BJ17}. 
The invariant may also be calculated in terms of valuations or filtrations.

\begin{thm}[{\cite[Theorems A,C]{BJ17}}]\label{t-delta}
We have
\begin{equation*}
\delta(X,\Delta)=\inf_E \frac{A_{X,\Delta}(E)}{ S(E)} = \inf_{v} \frac{A_{X,\Delta}(v)}{S(v)},
\end{equation*}
where the first infimum runs through all divisors $E$ over $X$ and the second through all $v\in \Val_X^*$ with $A_{X,\Delta}(v)< +\infty$. 
Furthermore, the limit $\displaystyle \lim_{m \to \infty} \delta_{mr} (X,\Delta)$ exists.
\end{thm} 

\begin{prop}[{\cite[Proposition 4.10]{BL18}}]\label{p-deltafilt}
We have 
\[
\d(X,\Delta) = \inf_{\cF} \frac{ \lct(X,\Delta; \fb_\bullet(\cF)) }{S(\cF)}
\]
where the infimum runs through all non-trivial linearly bounded filtrations of $R(X,L)$. 
\end{prop}
Combining Definition \ref{d-kstable} and Theorem \ref{t-delta}, we immediately see 
\begin{thm}[{\cite{FO16,BJ17}}]
A log Fano pair $(X,\Delta)$ is uniformly K-stable (resp., K-semistable) if and only if $\d(X,\Delta) >1$ (resp., $\geq 1$).
\end{thm}

While in Section \ref{s-uniK} we will use the definition of the $\d$-invariant in terms of $m$-basis type divisors, in Section \ref{s-kes} we will rely on its characterization in terms of valuations and filtrations.

\subsection{Normalized volume}\label{ss-normv}
Now, we discuss an invariant similar to the $\delta$-invariant, but defined in a local setting. 
The invariant was first introduced in \cite{Li18} and is closely related to  the K-semistability of log Fano pairs. 

Let $(X,\Delta)$ be an $n$-dimensional klt pair and $x\in X$ a closed point. The {\it non-archimedean link} of $X$ at $x$ is defined  as
$$\Val_{X,x} := \{ v\in \Val_{X} \, \vert\, c_X(v) = x  \, \} \subset \Val_X.$$

\begin{defn}[\cite{Li18}]\label{d-normvol}
The \emph{normalized volume function} $\hvol_{(X,\Delta),x}:\Val_{X,x}\to(0, +\infty]$
 is defined by
  \[
  \hvol_{(X,\Delta),x}(v)=\begin{cases}
            A_{(X,\Delta)}(v)^n\cdot\vol(v) & \textrm{ if }A_{(X,\Delta)}(v)<+\infty;\\
            +\infty & \textrm{ if }A_{(X,\Delta)}(v)=+\infty.
           \end{cases}
 \]
 The \emph{volume of the singularity} $(x\in (X,\Delta))$ is defined as
 \[
  \hvol(x, X,\Delta):=\inf_{v\in\Val_{X,x}}\hvol_{(X,\Delta),x}(v).
 \]
 The previous infimum is a minimum by the main result in \cite{Blu18}.
 \end{defn}

See \cite{LLX18} for a survey on the recent study of the normalized volume function, especially the guiding question,  the Stable Degeneration Conjecture (see \cite[7.1]{Li18} and \cite[4.1]{LLX18}). 

\subsubsection{Relation to K-stability}\label{ss-nvolkss}
The connection between the normalized volume function and K-semistability is via the cone construction first studied in \cite{Li17}.

 Let  $(X,\Delta)$ be a log Fano pair and $r$ a positive integer so that $L:=-r(K_X+\Delta)$ is a Cartier divisor.  Let  $(Z,\Gamma)$ denote the cone over $X$ with respect to the polarization $L$
 and $x\in Z$ denote the vertex. Specifically, $Z= \Spec(R)$, where $R=R(X,L)$ and $\Gamma$ is the  
 closure of the pullback of $\Delta$ via the projection map $Z \setminus \{x\} \to X$.

There is a natural map $X_{L} \to Z$, where $X_{L}: =\Spec_X   \big(\bigoplus_{p \geq 0} \cO_X(pL) \big)$ is the total space of the line bundle on $X$ whose sheaf of sections is $\cO_{X}(L)$. 
The map is an isomorphism over $ Z \setminus x $ and  the preimage of the vertex is  the zero section $X_{ \rm zs} \subset X_{L}$. We call $v_0 : = \ord_{X_{ {\rm zs}}}$ the \emph{canonical valuation}
over the cone.  

\begin{thm}[{\cite{Li17,LL16,LX16}}]\label{t-cone}
The canonical valuation $v_0$ is a minimizer of $\hvol_{(Z,\Gamma),x}$ if and only if $(X,\Delta)$ is K-semistable. 
\end{thm}

At first sight, using the normalized volume function to study the K-stability of log Fano pairs may seem indirect. 
However, this approach yields a number of new results (for example, see \cite{LX16, LWX18}). 
In this paper, the following key ingredient in the proof of Theorem \ref{t-cone} plays an important role in the proof of our main result.

Following  \cite{Li17,LX16}, let $E$ be a divisor over $X$ that arises on a proper normal model $\mu:Y\to X$. 
Consider the natural birational maps
\[
 Y_{L} \to X_{L} \to Z,\]
 where $Y_{L} := \Spec_Y \big( \bigoplus_{m \geq 0} \cO_Y(m \mu^*L) \big)$. 
 Let $Y_{ {\rm zs}} \subset  Y_{L}$ denote the zero section and $E_{\infty}$ the preimage of $E$ under the projection $Y_{L} \to Y$. 
 Setting 
 $v_t$ equal to the quasimonomial valuation with weights $(1,t)$ along $Y_{\rm zs}$ and $E_{\infty}$
 gives a ray of valuations
$$\{v_{t}\ |\  t\in [0,\infty) \} \subset \Val_{Z,x},$$
where  $v_0 = \ord_{X_{\rm zs} } $ and $v_\infty  = \ord_{ E_{\infty} }$. 
 When $k\in \mathbb{N}$, there exists a divisor $E_k$ over $Z$ satisfying $v_{\frac{1}{k}}=\frac{1}{k}\ord_{E_k}$.

By the formula for the log discrepancy of a quasimonomial valuation \cite[5.1]{JM12}, 
  \[
  A_{Z,\Gamma} (v_t) =   A_{Z,\Gamma} (\ord_{X_ {\rm zs} }) + t A_{Z,\Gamma}( \ord_{ E_{\infty} }) = r^{-1}+ tA_{X,\Delta}( \ord_E) 
 .\]
The valuation ideals are given by, for $t>0$, 
  \[ \fa_p(v_t) = \underset{m\geq 0 }{\bigoplus} {\cF_E}^{(p-m)/t}R_m \subseteq R
  \quad \text{ and } \quad
  \fa_p(v_0) = \underset{m\geq p } {\bigoplus} R_m  \subseteq R
  .\]
  
To see the previous formula holds,
 fix a uniformizer $\omega \in \cO_{Y,E}$ and a local section $s$ of  $\cO_{Y}(\mu^*L)$ that trivializes the sheaf at the generic point of $E$.
  The choice of $s$ induces a  rational map 
 $ Y_{L}  \dashrightarrow Y \times \mathbb{A}^1$
that is an isomorphism  at the generic point of $Y_{\rm zs} \cap E_{\infty}$. 
The birational transforms of  $Y_{\rm zs}$ and $E_{\infty}$ on $Y\times \mathbb{A}^1$ are
 $Y\times  0  $ and 
$E \times  \mathbb{A}^1$.

Fix  $f = \sum_{m \geq 0 } f_m \in R$. For $m$ such that $f_m \neq 0$, set $g_m : = f_m / \omega^{ \ord_E(f_m) }$, 
which, at the generic point of $E$, is a non-vanishing section of  $\cO_{Y}(\mu^*L)$. 
The image of $f$ in $\cO_{Y\times \mathbb{A}^1, E \times 0}$
equals  $\sum_m  \left( \frac{g_m}{s^m} \right)  \omega^{ \ord_E(f_m)}\pi^m$, where $\pi$ is the parameter for $\mathbb{A}^1$. 
Since $\pi$ and $\omega$ are local equations for  $Y\times 0$ and $E\times \mathbb{A}^1$ at the generic point of $E\times 0$ and $ \frac{g_m}{s^m}$ does not vanish at $E\times 0$,
$$v_t (f) : = \min_m \{ t \cdot \ord_E(f_m) + m \, \vert \, f_m \neq 0  \}  $$
and the formula for $\fa_{p}(v_t)$ follows.

By the calculation in \cite[(31, 32)]{Li17} (see also the proof of \cite[4.5]{LX16} or Lemma~\ref{l-dervolume})
\begin{eqnarray}\label{e-betader}
\frac{d}{ dt}\hvol(v_t)\big\vert_{t=0^+} = (n+1) \beta_{X,\Delta}(E)
\end{eqnarray}
This  equation  is a key input in our proof of  Theorem \ref{t-delta=1}.  More specifically, we will follow ideas from \cite{LWX18} and analyze directions along which the normalized volume function has derivative equal to zero.

\subsubsection{C. Li's derivative formula}
In the proof of Theorem \ref{t-main}, we will need a more general version of \eqref{e-betader}. The more general formula follows from the original argument in \cite{Li17}. 

Let $(X,\Delta)$ be an $n$-dimensional log Fano pair and $r\in \Z_{>0}$ so that $L:=-r(K_X+\Delta)$ is Cartier.
Set $R = R(X,L)$ and fix a linearly bounded filtration $\cF$ of $R$.

Associated to $\cF$, we define a collection of graded sequences of ideals of $R$.
For $t \in \R_{>0}$ and $j \in \Z_{>0}$, set
\[
\fb_{t,j}  := \underset{m \geq 0 }{\bigoplus} \cF^{{(j-m)}/{t} }R_m \subset  R \quad \text{ and } \quad \fb_{0,j} := \underset{ m \geq j}{\bigoplus} R_m \subset R
.\]
Note that $\fb_{t,\bullet}$ is a graded sequence of ideals of $R$ for each  for each $t\in \R_{\geq 0}$. Additionally, $\fb_{t,j}$ contains $ \oplus_{m\geq j} R_m$.

\begin{lem}\label{l-dervolume}
With the above notation, fix $A>0$ and set 
\[
f(t) =(r^{-1} +At)^{n+1} \mult(\fb_{t,\bullet})\]
for $t\in \R_{\geq 0}$. 
The following hold:
\begin{enumerate}
\item $\mult(\fb_{t,\bullet}) =  r^n (-K_X-\Delta)^n - (n+1) \int_{0}^{\infty} \vol (\cF R^{(x)}) \frac{ t\, dx}{(1+tx)^{n+2}}$;
\item $\frac{df}{dt} \vert_{t=0^+}  = (n+1) (-K_X-\Delta)^n \left( A- S(\cF) \right)$.
\end{enumerate}
\end{lem}

\begin{proof}
This follows from the argument in \cite[(18)-(25)]{Li17}. For the reader's convenience, we give a brief proof.
For $t\in \R_{>0 }$, we have
\begingroup
\allowdisplaybreaks
\begin{align*}
\mult(\fb_{t,\bullet })&=\lim_{j\to \infty}  \frac{(n+1)!}{j^{n+1}}\dim (R/\fb_{t,j})\\
&=\lim_{j\to \infty}  \frac{(n+1)!}{j^{n+1}}\sum^{\infty}_{m=0}\dim  (R_m/ \cF^{{(j-m)}/{t}} R_m ) \\
&=\lim_{j\to \infty}  \frac{(n+1)!}{j^{n+1}}\sum^{j}_{m=0}\left(\dim R_m - \dim  \cF^{{(j-m)}/{t}} R_m \right)\\
& = \vol(L)  - \lim_{j\to \infty}  \frac{(n+1)!}{j^{n+1}} \sum_{m=0}^j \dim \cF^{{(j-m)}/{t}} R_m .
\end{align*}
\endgroup
Statement (1) now follows from \cite[(25)]{Li17}, where $c_1=0$, $\alpha=\beta=\frac{1}{t}$. 

For (2), compute $$
\frac{df}{dt} \Big \rvert_{t=0^+} =(n+1) A r^{-n} \mult(\fb_{0,\bullet} )+ r^{-n-1} \frac{d}{dt} \left(\mult (\fb_{t, \bullet} )\right)\Big\rvert_{t=0^+}.$$
From (1), we know $\mult(\fb_{0,\bullet})=r^n(-K_X-\Delta)^n$
and
\begin{align*}
 \frac{d}{dt} \left( \mult(\fb_{t,\bullet} ) \right) \Bigr\rvert_{t=0^+}&=-
 (n+1)\int_0^\infty   \left( \vol(\cF R^{(x)})\left( \frac{ 1- (n+1) tx}{(1+tx)^{n+3} } \right) \right) \Bigr\rvert_{t=0^+}dx
.  \end{align*}
Since the latter simplifies to $-(n+1)\int^{\infty}_{0} \vol (\cF R^{(x)})dx$, (2) is complete.
\end{proof}

\section{Uniformly K-stable Fanos}\label{s-uniK}

In this section, we prove a special case of Theorem \ref{t-main} for uniformly K-stable Fano varieties. We will then apply the result to study the moduli functor $\mathcal{M}^{\rm uKs}_{n,V}$. 

\subsection{Separatedness result} 

The following result is a special case of Theorem \ref{t-main} and will be reproved in Section \ref{s-kes}. 
We present its proof independently, since the following argument is simpler than the proof in Section \ref{s-kes}.

\begin{thm}\label{t-sep}
Let $\pi: (X,\Delta) \to C$ and $\pi': (X',\Delta')\to C$ be $\Q$-Gorenstein families of log Fano pairs over a smooth pointed curve $0\in C$.
Assume there exists an  isomorphism
\[
\phi: (X,\Delta) \times_C C^\circ  \to (X',\Delta') \times_C C^\circ
\] 
over $C^\circ : = C\setminus 0 $.
 If $(X_0,\Delta_0)$ is uniformly K-stable and $(X_0',\Delta_0')$ is K-semistable, then $\phi$ extends to an isomorphism  $(X,\Delta) \simeq  (X',\Delta')$ over $C$.
\end{thm}

The proof of Theorem \ref{t-sep} follows from properties of the $\delta$-invariant and the following birational geometry fact. 

\begin{prop}\label{p-filling}
Let $\pi: (X,\Delta) \to C$ and $\pi': (X',\Delta')\to C$ be $\Q$-Gorenstein families of log Fano pairs over a smooth pointed curve $0\in C$.
Assume there exists an  isomorphism
\[
\phi: (X,\Delta) \times_C C^\circ  \to (X',\Delta') \times_C C^\circ
\] 
over $C^\circ : = C\setminus 0$.
If there exist effective horizontal\footnote{A $\Q$-divisor on $X$ or $X'$ is called \emph{horizontal} if its support does not contain a fiber of the map to $C$.} $\Q$-divisors $D$ and $D'$ on $X$ and $X'$ satisfying
\begin{itemize}
\item[(1)] $ D \sim_{\Q,C} -K_X-\Delta$ and  $ D' \sim_{\Q,C} -K_{X'}-\Delta'$ , 
\item[(2)] $D$ is the birational transform of $D'$, and
\item[(3)] $(X_0, \Delta_0+D_0)$ is klt and $(X_0' , \Delta'_0+D_0')$ is lc, 
\end{itemize}
then $\phi$ extends to an isomorphism $(X,\Delta) \simeq (X',\Delta')$ over $C$. 
\end{prop}

The above proposition is well known to experts and follows from the separatedness of the moduli functor of klt log Calabi-Yau pairs (e.g. see \cite[Theorem 4.3]{Oda12}, \cite[Theorem 5.2]{LWX14}). 
For the convenience of the reader, we prove the result.

\begin{proof}
Fix a  common log resolution $\widehat{X}$ of $(X,\Delta)$ and $(X',\Delta')$ 
\begin{center}
\begin{tikzcd}[column sep=scriptsize, row sep=scriptsize]
  & \widehat{X} \arrow[dr, "\psi'"] \arrow[dl,swap,"\psi"]& \\
  X  \arrow[rr, dashrightarrow, "\phi"] & & X'
   \end{tikzcd}   \end{center}
and write $\widetilde{X}_0$ and $\widetilde{X}'_0$ for the birational transforms of $X_0$ and $X'_0$ on $\widehat{X}$. 
   
First, assume $\widetilde{X}_0 =\widetilde{X}'_0$. This equality implies $\phi: X \dashrightarrow X'$ is an isomorphism in codimension one. 
   Thus, $\phi$ induces an isomorphism 
   $$\pi_*\cO_{X}(-m(K_X+\Delta ))  \simeq \pi'_* \cO_{X'} (-m( K_{X'}+\Delta'))$$ for all 
 $m\in \N$. Since
 \[
 X = \Proj_C \bigoplus_{m \geq 0 } \pi_*\cO_{X}(-m(K_X+\Delta )) \text{ and }
  X' = \Proj_C \bigoplus_{m \geq 0 } \pi'_*\cO_{X'} (-m( K_{X'}+\Delta')) 
 ,\]
 we conclude $\phi$ extends to an isomorphism over $C$.
 
We now assume  $\widetilde{X}_0 \neq  \widetilde{X}'_0$ and aim for a contradiction.   
Write
\[
  K_{\widehat{X}}+   \psi_*^{-1}(\Delta+D)  = \psi^*(K_X+\Delta+ D) +  a   \widetilde{X}'_0 +P\]
and 
\[
\,\,\,K_{\widehat{X} }+   {\psi'_*}^{-1}(\Delta'+D') =\psi'^*(K_{X'}+\Delta'+ D') +  a'   \widetilde{X}_0 +P'\]
where the components of $\Supp(P) \cup \Supp(P')$ are both $\psi$ and $\psi'$-exceptional.
By assumption (2), $P-P'$ is supported on $\widehat{X}_0$. 

Inversion of adjunction and our assumption that $(X_0, \Delta_0+D_0)$ is klt
imply $(X, \Delta+D+ X_0) $ is plt in a neighborhood of $X_0$. 
Hence, 
\[
-1< a(   \widetilde{X}'_0, X, \Delta+D+X_0) =a- \ord_{\widetilde{X}'_0} ( X_0 ).\]
Since $\ord_{ \widetilde{X}'_{0}} ( X_0) = 1$, we see $a >0$.  
The same argument, but with the assumption that  $(X_0', \Delta'_0+D_0')$ is lc
implies $a'\geq 0$. 

Observe
$a  \widetilde{X}'_0  - a'  \widetilde{ X}_0 + (P-P') \sim_{\Q,C} 0 $,
since 
$K_X+ \Delta+D \sim_{\Q,C} 0 $, $K_{X'}+ \Delta'+D' \sim_{\Q,C} 0$, and 
$\psi_*^{-1}(D+\Delta)={\psi'_*}^{-1}(D'+\Delta')$.
Therefore, there exists a rational number $c$ so that  
\[
a  \widetilde{X}'_0  - a'  \widetilde{ X}_0 + (P-P') \sim_{\Q,C} c \psi^*(X_0).\]
Comparing the coefficients of $\widetilde{X}'_0$ on the two sides implies $c>0$, 
while comparing the coefficients of  $\widetilde{X}_0$ implies $c\leq 0$. 
This is a contradiction.\end{proof}

\begin{lem}\label{l-fibebasis}
Keep the notation and setup of Theorem \ref{t-sep}. 
If $m\in \Z_{>0}$ is sufficiently divisible, then there exist  effective horizontal $\Q$-divisors $B$ and $B'$  on $X$ and $X'$ such that
\begin{itemize}
\item[(1)] $B \sim_{\Q,C} -K_{X}-\Delta$ and $B' \sim_{\Q,C} -K_{X'}-\Delta'$,
\item[(2)] $B$ is the birational transform of $B'$, and
\item[(3)]  $B_0$ and $B'_0$ are $m$-basis type with respect to $(X_0,\Delta_0)$ and $(X'_0,\Delta'_0)$.
\end{itemize}
\end{lem}

\begin{proof}
Fix a positive integer $m$ so that $L:=-m(K_X-\Delta)$ and $L':=-m(K_{X'}-\Delta')$ are Cartier
and $\pi_* \cO_X (L)$ and $\pi'_*\cO_{X'}(L)$ are nonzero. 
Since $H^i(X_t, \cO_{X_t}(L_t))$ and $H^i(X'_t, \cO_{X'_t}(L'_t))$ are zero for all $i>0$ and $t\in C$ by Kawamata-Viehweg vanishing, $\pi_*\cO_{X} (L)$ and $\pi'_*\cO_{X'}(L')$ are vector bundles. Furthermore, the sheaves satisfy cohomology and base change.

Now, the birational map $\phi$ induces a map from local sections of $\pi_* \cO_{X}(L)$ to rational sections of $\pi'_* \cO_{X'} (L')$. After twisting by $dX_0'$, where $d\gg0$, we get a morphism 
\[
 \pi_*\cO_{X}(L)   \to \pi'_*\cO_{X'}(L' + d X_0')\]
 that is an isomorphism away from $0\in C$. 
 Tensoring by $\cO_{C,0}$ gives a morphism
 \[
\varphi :  \pi_*\cO_{X}(L)  \otimes_{\cO_{C}} \cO_{C,0} \to \pi'_*\cO_{X'}(L'+ d X_0')\otimes_{\cO_C} \cO_{C,0},\]
of locally free $\cO_{C,0}$ modules that is an isomorphism after tensoring with $K(C)$. 

 Write $t$ for the uniformizer of $\cO_{C,0}$.
  Since $\cO_{C,0}$ is a principal ideal domain, there exist bases $\{ s_1,\ldots, s_{N} \}$ and $\{s_1', \ldots, s'_{N} \}$ for the above free modules so that the transformation matrix is diagonal. Hence, for each $1 \leq i \leq N$, there exist  $p_i \in \Z_{\geq 0}$ and $a_i \in \cO_{C,0}^\times$ so that $\varphi (s_i ) = a_i t^{p_i}   s_i'$.

For a sufficiently small neighborhood $0\in U\subset C$, we may extend each $s_i$ to a section $\tilde{s}_i \in \pi_*\cO_{X}(L)(U)$ and  $s_i'$ to a section $\tilde{s}_i' \in \pi'_*\cO_{X'}(L')(U)$. Let $B$ and $B'$ denote the closures of
 \[
 \frac{1}{mN}\big( \{  \tilde{s}_1 = 0 \} + \cdots + \{ \tilde{s}_{N} =0 \} \big) \text{ and }
  \frac{1}{mN}\big( \{  \tilde{s}_1' = 0 \} + \cdots + \{ \tilde{s}_{N}' =0 \} \big)
 .\]
 in $X$ and $X'$.
 By construction, $B_0$ and $B_0'$ are both $m$-basis type divisors and $B$ is the birational transform of $B'$. \end{proof}

\begin{proof}[Proof of Theorem \ref{t-sep}]
Since $X_0$ is uniformly K-stable and $X_0'$ is K-semistable, 
$$\delta(X_0,\Delta_0) >1\qquad \mbox{and} \qquad \delta({X_0'},\Delta'_0)\geq 1.$$ Hence, we may choose $0<\e \ll 1$ so that 
\begin{equation}\label{eq:deltine1}
\frac{1-\e}{\delta({X_0},\Delta_0)} + \frac{\e}{ \alpha(X_0,\Delta_0)} < 1,
\end{equation}
where $\alpha(X_0,\Delta_0)$ is Tian's $\alpha$-invariant, i.e.
$$\alpha(X_0,\Delta_0)=\inf \{\ \lct(X_0,\Delta_0;D)\  |\  0\le D\sim_{\bQ} -K_{X_0}-\Delta_0\}.$$
Next, choose a positive integer $M$ so that
\begin{equation}\label{eq:deltine2}
\frac{1-\e}{\delta_{m}({X_0},\Delta_0)} + \frac{\e}{ \alpha(X_0,\Delta_0) }< 1 \text{ and } \delta_{m} ({X_0'},\Delta_0') >1-\e
\end{equation}
for all positive integers $m$ divisible by $M$. Such a choice is possible by \eqref{eq:deltine1}, the inequality $\delta( X'_0, \Delta'_0)\geq 1$, and the fact that $\delta$ is a limit.

Now, fix a positive integer $m$ divisible by $M$ so that the conclusion of Lemma \ref{l-fibebasis} holds for $m$ and $-m (K_{X'} +\Delta' )$ is relatively base point free over $C$. Hence, we may find $\Q$-divisors $B \sim_{\Q,C} -K_X-\Delta$
and $B' \sim_{\Q,C}-K_{X'} -\Delta'$ satisfying the conclusion of Lemma \ref{l-fibebasis} for $m$. Since $B_0$ and $B_0'$ are $m$-basis type, 
$$\lct(X_0,\Delta_0;B_0) \geq \delta_{m} ({X_0},\Delta_0) \mbox{\ \ \ and \ \ \ }\lct(X'_0,\Delta'_0;B_0') \geq \delta_{m} ({X_0'},\Delta_0')>1-\e.
$$ 
The latter implies $(X'_0, \Delta'_0+ (1-\e) B_0')$ is lc. 

Since $-m(K_{X'}+\Delta')$ is relatively base point free over $C$, after shrinking $C$ in a neighborhood of $0$, we may apply \cite[Lemma 5.17]{KM98} to find an effective divisor $G' \in |-m (K_{X'}+\Delta')|$ in general position so that  $(X'_0, \Delta'_0+(1-\e) B_0' +(\e/m)  G_0')$ remains lc. Write $G$ for the birational transform of $G'$ on $X$. Note that $G \sim_{\Q,C} -m(K_{X}+\Delta)$, since the statement holds over $C^\circ$. Thus, $\lct(X_0,\Delta_0;(1/m)G_0) \geq  \alpha(X_0,\Delta_0)$. 

Now, consider the divisors
\[
D :=(1-\e) B + \frac{\e}{m} G \hspace{.5 cm}
\text{ and }
\hspace{.5 cm}
D' : =(1-\e) B'+  \frac{\e}{m} G'.\]
Observe that $D \sim_{\Q,C} -K_{X}-\Delta$ and  $D' \sim_{\Q,C} -K_{X'}-\Delta'$. As mentioned above, $(X_0, \Delta'_0+D_0')$ is lc. 
Additionally, the pair $(X_0,\Delta_0+ D_0)$ is klt. Indeed, since 
$$1/ \lct(D+F) \leq 1/\lct(D) + 1/\lct(F)$$
for any two effective $\Q$-Cartier $\Q$-divisors $D$ and $F$ on a klt pair, we know 
\begin{align*}
\frac{1} {\lct(X_0,\Delta_0;D_0)}
&\leq \frac{1}{ \lct( X_0,\Delta_0;(1-\e)B_0)}
+ \frac{1}{\lct( X_0,\Delta_0;(\e/m)G_0)} \\
& \leq \frac{1-\e}{\delta_{m}(X_0,\Delta_0) } + \frac{\e}{ \alpha({X_0},\Delta_0) }
 \end{align*}
 which is $<1$ by \eqref{eq:deltine2}.
 Proposition \ref{p-filling} now implies $\phi$ extends to an isomorphism.
\end{proof}

\begin{rem}
If $(X_0,\Delta_0)$ and $(X'_0,\Delta'_0)$ are only assumed to be K-semistable, then they are not necessarily isomorphic (but are S-equivalent by Theorem \ref{t-main}). Therefore, we do not expect the the above strategy to be useful in this more general case. 
\end{rem}

 Recall, if $(X,\Delta)$ is a log Fano pair, then $\Aut(X,\Delta)$ is the closed subgroup of $\Aut(X)$ defined by
\[
\Aut(X,\Delta): = \{  g \in \Aut(X) \, \vert \, g^* \Delta = \Delta \}
.
\]
The following result is an immediate corollary of Theorem \ref{t-sep} and a special case of Corollary \ref{c-aut}.

\begin{cor}\label{c-aut2}
Let $(X,\Delta)$ be a log Fano pair. If $(X,\Delta)$ is uniformly K-stable, then ${\rm Aut}(X,\Delta)$ is finite. 
\end{cor}

\begin{proof}
Since $\Aut(X,\Delta)$ is a linear algebraic group, it is affine. 
To conclude that $\Aut(X,\Delta)$ is finite, it suffices to show that it is proper.
To see the properness, consider a map $g:C^\circ \to \Aut(X,\Delta)$, where $0 \in C$ is a smooth pointed curve and $C^\circ = C\setminus 0$. The map $g$ induces an isomorphism 
$$(X\times C,\Delta\times C)\times_C C^\circ   \to (X\times C,\Delta\times C)\times_C C^\circ 
$$
over $C^\circ$. 
By applying Theorem \ref{t-sep} to the above isomorphism, we see $f$ extends to a map $\overline{g}: C\to \Aut(X,\Delta)$. Hence, $\Aut(X,\Delta)$ is proper, and the proof is complete.
\end{proof}

In \cite[Corollary E]{BHJ16}, it is shown that the polarized automorphism group of a  uniformly K-stable  polarized manifold $(X,L)$ is finite. Their proof uses analytic tools.

\begin{rem}
Our proofs of Theorem \ref{t-sep} and Corollary \ref{c-aut2} extend to the case of polarized klt pairs $ (X,\Delta;L)$ (that is, $(X,\Delta)$ is a projective klt pair and $L$ an ample $\Q$-Cartier divisor on $X$) such that $K_{X}+\Delta+L$ is nef and $\delta(X,\Delta;L)>1$.  
\end{rem}

\subsection{Moduli spaces}

\begin{proof}[Proof of Corollary \ref{c-moduli}]
As previously mentioned, the result relies on \cite{Jia17,BL18} and Theorem \ref{t-sep}. 
Indeed, \cite{Jia17} (see also \cite{Che18} or \cite[6.14]{LLX18}) states that the set of varieties $\mathcal{M}_{n,V}^{\rm uKs}(k)$ is bounded. Hence, there exists a positive integer $M$ so that $-MK_{X}$ is a very ample Cartier divisor for all $X \in \mathcal{M}_{n,V}^{\rm uKs}(k)$.
Furthermore, the set of Hilbert functions $m\mapsto \chi\big( \omega_{X}^{[-mM]}\big)$ with $X \in  \mathcal{M}_{n,V}^{\rm uKs}(k)$ is finite. 

For such a Hilbert function $h$, consider the subfunctor $\mathcal{M}_{h}^{\rm uKs} \subset \mathcal{M}_{n,V}^{\rm uKs}$
parameterizing uniformly K-semistable $\Q$-Fano varieties with Hilbert function $h$. Note that $\mathcal{M}_{n,V}^{\rm uKs}= \coprod_{h} \mathcal{M}_{h}^{\rm uKs}$. Set $N:=h(1)-1$, and let ${\rm Hilb}(\PP^N)$ be the Hilbert scheme parameterizing closed subschemes of $\PP^N$ 
with Hilbert polynomial $h$. Write $X\to  {\rm Hilb}(\PP^N)$ for the corresponding universal family.

Now, let $U \subset  {\rm Hilb}(\PP^N)$ denote the open subscheme parameterizing normal, Cohen-Macaulay varieties.
By \cite[3.11]{HK04},
 there is a locally closed subscheme $V \subset U$ such that a map
$T\to U$ factors through $V$ if and only if there is an isomorphism $\omega_{X_T/ T}^{[-M]} \simeq \mathcal{L}_T \otimes \cO_{X_T}(1)$, 
where $\mathcal{L}_T$ is the pullback of a line bundle from $T$.
By applying \cite{BL18}  to the normalization of $V$, we see
$$V': =\{t \in V \, \vert \, X_{\overline{t}}  \text{ is a uniformly K-stable $\Q$-Fano variety} \}$$
is open in $V$. 
Finally,  we apply \cite[25]{Kol08} or \cite{AH11} to find a locally closed decomposition  $W\to V'$ such that a morphism $T \to V'$ factors through $W$ if and only if $X_{T} \to X$ satisfies Koll\'ar's condition. 

As a consequence of the above discussion, $\mathcal{M}^{\rm uKs}_{h}\simeq [W/{\rm PGL}(N+1)]$. 
Theorem \ref{t-sep} implies $\mathcal{M}^{\rm uKs}_{h}$ is  a separated Deligne-Mumford stack. 
Furthermore, we may apply \cite{KM97} to see  $\mathcal{M}^{\rm uKs}_{h}$  has a coarse moduli space ${M}^{\rm uKs}_{h}$, which is a separated algebraic space.
\end{proof}

\section{Places  computing the $\delta$-invariant}\label{s-delta1}

In this section, we will study the cases when  valuations, ideals, and $\Q$-divisors compute the $\d$-invariant. 
The results proved here are related to Conjecture \ref{conj-special} and will be used in the proof of Theorem \ref{t-main}.

\subsection{Divisors computing $\delta$}\label{ss-divisorial}

\begin{thm}\label{t-delta=1}
Let $(X,\Delta)$ be a K-semistable log Fano pair. If $E$ is a divisor over $X$ satisfying  $\delta(X,\Delta) = \frac{A_{X,\Delta}(E)}{S(E)} =1$,
 then $E$ is dreamy and induces a non-trivial special test configuration $(\cX,\cD)$ such that $\Fut(\cX,\cD)=0$. In particular, $(X,\Delta)$ is not K-stable.  
\end{thm}

The proof follows an argument in \cite[Section 3.1]{LWX18}.
The argument will be used again in the proof of Lemma \ref{l-Wkextract} in a relative setting.

\begin{proof}
Fix a positive integer $r$ so that $L:=-r(K_X+\Delta)$ is a Cartier divisor and set $R=R(X,L)$. 
Consider the cone
$(Z,\Gamma)$ over $(X,\Delta)$ with respect to the polarization $L$.

The divisor $E$ over $X$ induces a ray of valuations
$$ \{v_t \, \vert \, t\in [0,\infty) \}\subset \Val_{Z,x}$$ (see Section \ref{ss-nvolkss}).
For $k\in \Z_{>0}$, there is a divisor $E_k$ over $Z$ so that  
 $v_{\frac{1}{k}}=\frac{1}{k}\ord_{E_k}$.
 By \eqref{e-betader},
\[
 \frac{d }{dt}  \hvol(v_t) \big\vert_{t=0^+}= (n+1) \beta_{X,\Delta}(E).
\]
 Since $A_{X,\Delta}(E)-S(E)=0$,  we know $\beta_{X,\Delta}(E)=0$.
Defining $f(t):=\hvol(v_t)$, a Taylor expansion gives 
 $$f(t)=f(0)+O(t^2) \mbox {\ \ \ \ \ \ \ for $0\le t \ll 1$}.$$ 

For a fixed positive integer $k$, 
set 
$$\fa_{k,\bullet}:=\fa_\bullet(\ord_{E_k}) \quad
\text{ and } \quad c_k:=\lct( Z,\Gamma; \fa_{k,\bullet}).$$
Note that $c_k\le A_{Z,\Gamma}(E_k)$ by  \eqref{e-lctA}. 
This implies
$$f(0)\le c^{n+1}_k\cdot \mult(\fa_{k,\bullet})\le f\bigg(\frac 1 k\bigg)=A_{Z,\Gamma}(E_k)^{n+1}\cdot \mult(\fa_{k,\bullet}), $$
where the first inequality follows from \cite[7]{Liu18} and the assumption that $(X,\Delta)$ is K-semistable.
Therefore,
\[
\bigg(  \frac{1}{ 1+ O(1/k^2)} \bigg)^{\frac{1}{n+1} }
\leq 
\left( \frac{ f(0)}{ f(1/k)} \right)^{\frac{1}{n+1} } \leq \frac{c_k}{A_{Z,\Gamma}(E_k)}  \leq 1 
.\]
Since  $(1+O(1/k^2))^{1/(n+1)}$ is of the order $1+ O(1/k^2)$,
we see
$$1-O\left(\frac{1}{k^2} \right)=\frac{c_k}{A_{Z,\Gamma}(E_k)}\le 1.$$ 
Using that  $A_{Z,\Gamma}(E_k)=kr^{-1} + A_{X,\Delta}( E)$,
$$
\lim_{k \to \infty} 
 \left( A_{Z, \Gamma}(E_k) -c_k \right)  
 =
 \lim_{k \to \infty} 
 \left( 
 A_{Z,\Gamma}(E_k)\left(
1-\frac{c_k}{A_{Z,\Gamma}(E_k)}
\right)  \right)= 0.
$$
Hence, we may fix $k\gg 0$ so that $A_{Z,\Gamma}(E_k)-c_k<1$. 

By  Proposition \ref{p-extract}, there exists a proper birational morphism $\mu_k\colon Z_{k}\to Z$ such that $E_k \subset Z_k$  and $-E_k$ is ample over $Z$.
Therefore, $\bigoplus_{p\geq 0 } \mu_{k*} \cO_{Z_k}(-pE_k )$ is a  finitely generated $\cO_{Z}$-algebra. 
Since  $\mu_{k*}( \cO_{Z_k}(-pE_k ))  = \fa_p(kv_{1/k})$,  the latter implies
\[
 \bigoplus_{p \in \N} \bigoplus_{m \in \N} \cF_E^{p-mk} R_m 
\]
is a
finitely generated $R$-algebra. 
Therefore, ${\rm Rees}(\cF_E)$ is finitely generated as well and $E$ is dreamy. 

Let $(\cX,\cD)$ denote the test configuration induced by $\cF_E$. 
The test configuration is normal and non-trivial  \cite[3.8]{Fuj17b}
and $\Fut(\cX,\cD)$
 is a multiple of $A_{X,\Delta}(E) - S(E)$  \cite[6.12]{Fuj16}, which is zero. 
 We conclude $(\cX,\cD)$ is special, 
 since otherwise there would exist a test configuration of $(X,\Delta)$
  with negative Futaki invariant \cite[1]{LX14}. 
\end{proof}



An immediate corollary to Theorem \ref{t-delta=1} is the following strengthening of \cite[1.6]{Fuj16} and \cite[3.7]{Li18}. The result was expected in the arXiv version of \cite{Li18}. 
\begin{cor}\label{c-alld}
A log Fano pair $(X,\Delta)$ is K-stable if and only if  $\beta_{X,\Delta}(E)>0$ for any divisor $E$ over $X$. 
\end{cor}

\begin{proof}
Theorem \ref{t-delta=1} implies the forward implication. The reverse implication was shown in \cite[1.6]{Fuj16} and \cite[3.7]{Li18}.
\end{proof}

\subsection{Ideals computing $\delta$}\label{ss-intersection}

Let $(X,\Delta)$ be a log Fano pair and $\fa \subsetneq \cO_X$ a nonzero ideal. 
Write $\pi:Y \to X$ for the normalized blowup of $X$ along $\fa$ and $E$ for the effective Cartier divisor on $Y$ such that $\fa \cdot \cO_Y = \cO_Y(-E)$. 
We set
\[
S(\fa) : =
\frac{1}{ \vol(-K_X-\Delta)} \int_0 ^{ +\infty} \vol( \pi^*(-K_X-\Delta) - t E) \, dt.
\]

\begin{prop}\label{p-Sideal}
If $(X,\Delta)$ is a log Fano pair and $\fa \subsetneq \cO_{X}$ a nonzero ideal, then
\begin{equation}\label{e-deltaideal}
  \frac{ \lct(X,\Delta;\fa) }{S(\fa)}\geq \delta(X,\Delta)
.\end{equation}
Furthermore, write $\pi:Y\to X$ for the normalized blowup of $\fa$ and $E$ for the Cartier divisor on $Y$ such that $\fa \cdot \cO_Y = \cO_Y(-E)$. If  \eqref{e-deltaideal} is an equality, then $\Supp(E)$ is a prime divisor and computes $\delta(X,\Delta)$. 
\end{prop}

The above proposition is an analog of \cite[Theorem 3.11]{LX16} for the $\delta$-invariant and is similar to \cite[Corollary 3.22]{Fuj17}.

\begin{proof}
Choose a divisor $F$ over $X$ computing $\lct(X,\Delta;\fa)$. 
By \cite{BCHM10}, there is an extraction $\rho: X_F \to X$ of $F$.
Set $p:= \ord_F(\fa)$. 
Hence, $A_{X,\Delta}(F) / p = \lct(X,\Delta;\fa)$ and $\fa^k \cdot \cO_{X_F} \subseteq \cO_{X_F}(-kpF)$ for all $k \in \N$. 

By the previous inclusion, if we set $L:=-K_X-\Delta$, then
\[
\vol(\pi^*L-tE) \leq \vol( \rho^*L- tpF) \]
for all $t\in\R_{\geq 0}$.
Hence, $S(\fa) \leq p^{-1} S(F)$, and we see
\[
\frac{\lct(\fa)}{S(\fa)}  \geq \frac{ A_{X,\Delta}(F) }{ S(F)} .\]
Since $A_{X,\Delta}(F) / S(F) \geq \delta(X,\Delta)$, \eqref{e-deltaideal} holds.

Now assume \eqref{e-deltaideal} is an equality. In this case, the above argument implies $F$ computes $\delta(X,\Delta)$.
To finish the proof, it suffices to show $Y=X_F$ and $\Supp(E)=F$.

Fix a positive integer $k$ so that $-kpF$ is Cartier and choose an ideal $\fc\subseteq \cO_{X_F}$ such that 
\[
\fa^k \cdot \cO_{X_F}  = \fc  \cdot  \cO_{X_F}(-pkF).\] 
Write $\tau:Z \to X_F$ for the normalized blowup of $X_F$ along $\fc$  
and $G$ for the Cartier divisor on $Z$ such that $\fc \cdot \cO_{Z} = \cO_Z(-G)$. 
Since $Z$ is normal and 
$$\fa^k \cdot \cO_Z= \left( \fc  \cdot \cO_{X_F}(-pkF) \right)\cdot \cO_Z = \cO_Z(-pk\tau^*(F)-G)$$ is locally free, 
$\rho \circ \tau $ factors through $\pi$:
\begin{center}
\begin{tikzcd}[column sep=scriptsize, row sep=scriptsize]
  &Z  \arrow[dl,swap, "\sigma"] \arrow[dr,,"\tau"]& \\
  Y  \arrow[dr,swap,  "\pi"] & &  X_F  \arrow[dl, "\rho"] \\
  & X&
   \end{tikzcd}  . \end{center}
Additionally, $\sigma^* (E) = \tau^* (pF) + k^{-1} G$.

If we can show $\fc= \cO_{X_F}$, the proof will be complete. 
Indeed, if $\fc= \cO_{X_F}$, then $\tau$ is an isomorphism and $\sigma^*E = p F$. 
But, since $\sigma^*E = p F$ is anti-ample over $X$, $\sigma$ must also be an isomorphism and we are done. 

We claim that if $\fc\subsetneq  \cO_{X_F}$, then
\[
\vol(\pi^*L-tE) < \vol( \rho^*L- tpF) 
\]
for $0 <t \ll 1$ and, thus, $S(\fa)< (1/p) S(F)$. Since, we will then have
\[
\delta(X,\Delta) \leq \frac{A_{X,\Delta}(F)}{S(F)} <  \frac{ \lct(\fa)}{S(\fa)} = \delta(X,\Delta),\]
a contradiction will be reached. 

To prove the above claim, fix $0 <\e \ll 1/k$ so that $H:=  p\tau^* F + \e G$
 is anti-ample over $X$. Note that by our choice of $\e$, we also have 
 \[   \vol( \pi^*L- tE)= 
 \vol(  \tau^* (\rho^*L)- t \sigma^*E ) \leq \vol( \tau^* ( \rho^* L)- tH). 
 \] Therefore, it  suffices to show 
 \[
 \vol( \tau^*( \rho^* L)- tH) <  \vol(  \rho^* L- tpF)
 \]
 for $0 < t \ll 1$. 
 
Fix $0 < t \ll 1$ so that both  $A_t := \rho^*L-tpF$ and $B_t:= \tau^* (\rho^*L)- t H$ are both ample. 
Following an argument in \cite[3.3]{Fuj17}, we note that for $0 \leq i \leq n-1$, 
\begin{align*}
0 & \leq  \e t G \cdot  ( \tau^*A_t )^{i} \cdot B_t^{n-i-1}   \\
   &  =   (\tau^*A_t- B_t) \cdot ( \tau^*A_t )^{i} \cdot B_t^{n-i-1}   ,
\end{align*}
since $G$ is effective, $ \tau^*A_t$ is nef, and $B_t$ is ample. Additionally, 
\[
0< (\tau^*A_t- B_t)  \cdot B_t^{n-1} .  
\]
We now see 
\begin{align*}
0 &<\sum_{i=0}^{n-1}  \left(
\left( \tau^* A_t - B_t \right)  \cdot  ( \tau^*A_t )^{i} \cdot B_t^{n-i-1}  \right)  \\
& =
(\tau^* A_t)^n -(B_t)^n  \\
& = 
\vol(  \rho^* L- tpF) - \vol( \tau^*( \rho^* L )- t H) , \end{align*}
and conclude $\vol(\rho^*L-t p F)<  \vol( \tau^*( \rho^* L )- t H)$ for  $0<t\ll1$.
\end{proof}

\subsection{$\Q$-divisors computing $\delta$}\label{ss-qdivisor}

Let $(X,\Delta)$ be a log Fano pair, $\mu:Y \to X$ a proper birational morphism with $Y$ normal, 
and $E$ an effective $\Q$-Cartier $\Q$-divisor on $Y$ such that $-E$ is $\mu$-ample. 
We set 
$\fa_p(E) := \mu_* \cO_Y(-\lceil pE \rceil ) \subseteq \cO_X$
and  
\[S(E) := \frac{1}{(-K_X-\Delta)^n} \int_0^\infty \vol(\mu^*(-K_X-\Delta)- tE)\, dt.\]

\begin{prop}\label{p-Qdivcomputing}
With the above notation, 
we have 
\begin{equation}\label{e-Excdelta}
\frac{ \lct(X,\Delta; \fa_\bullet(E) )}{S(E)} \geq \delta(X,\Delta). 
\end{equation}
Furthermore, if \eqref{e-Excdelta} is an equality, then $\Supp(E)$ is a prime divisor. 
\end{prop}

The statement is a consequence Proposition \ref{p-Sideal} and the following elementary lemma.

\begin{lem}\label{l-antiample}
Let $\mu:Y \to X$ be a proper birational morphism of normal varieties and $E$ an effective $\Q$-Cartier $\Q$-divisor on $Y$ such that $-E$ is $\mu$-ample. 
Set 
$$\fa_{p}(E) : = \mu_* \cO_{Y}( -\lceil pE\rceil) \subseteq \cO_X.$$ 
If $p\in \Z_{>0}$ is sufficiently divisible, then
\begin{enumerate}
\item $Y$ is the blowup of $X$ along $\fa_{p}(E)$,
\item $\fa_{p}(E) \cdot \cO_Y = \cO_Y(-pE)$, and
\item  $ \left( \fa_{p}(E)\right)^\ell  = \fa_{p\ell }(E)$ for all $\ell \in \Z_{>0}$.
\end{enumerate}
\end{lem}

\begin{proof}
Since $-E$ is ample over $X$, 
$\bigoplus_{m \in \N} \fa_m(E)$ is a finitely generated $\cO_X$-algebra and 
$
Y \simeq  \Proj_X \big( \bigoplus_{m \in \N} \fa_m(E)\big).
$
The former statement implies that if $p\in \Z_{>0}$ is sufficiently divisible, 
then the $p$-th Veronese,
$\bigoplus_{m \in \N} \fa_{pm}(E)$, is finitely generated in degree 1. 
Hence, (1) and (3) are complete.
For (2), observe that the natural map $\mu^* \mu_* \cO_X(- pE ) \to \cO_{Y}(-pE)$ is surjective for $p\in \Z_{>0}$ sufficiently divisible, since $-E$ is $\mu$-ample.
\end{proof}

\begin{proof}[Proof of Proposition \ref{p-Qdivcomputing}]
Fix $p\in \Z_{>0}$ satisfying (1)-(3) of Lemma \ref{l-antiample} and set $\fa:= \fa_p(E)$. 
By (1) and (2),   $p \cdot S(\fa )=  S(E)$.
By (3)
\[
\lct(X,\Delta;\fa_\bullet(E)) := \lim_{m \to \infty}  \left( mp \cdot  \lct(X,\Delta; \fa_{pm}(E) )  \right) = p \cdot \lct(X,\Delta;\fa).
\] 
The result now follows immediately from Proposition \ref{p-Sideal}. 
\end{proof}

\section{Constructing the S-equivalence}\label{s-kes}

In this section, we prove Theorem \ref{t-main}. In Section \ref{ss-filt} we will construct filtrations of 
$$R=\bigoplus_{m\in \N} H^0(X_0, -mr(K_{X_0}+\Delta_0))\qquad\mbox{and} \qquad R'=\bigoplus_{m\in \N }H^0(X'_0, -mr(K_{X'_0}+\Delta'_0)),$$
whose associated graded rings are isomorphic.
Then in Section \ref{ss-twodeg}, we concentrate on proving that these filtrations and their associated graded rings are finitely generated.

\subsection{Filtrations induced by degenerations} \label{ss-filt}

Let 
\[
\pi: (X,\Delta) \to C
\quad
\text{ and }
\pi': (X',\Delta')\to C\]
 be $\Q$-Gorenstein families of $n$-dimensional log Fano pairs over a smooth pointed curve $0 \in C$.  
Assume there exists an isomorphism 
\[
\phi: (X,\Delta) \times_C C^\circ  \to (X',\Delta') \times_C C^\circ
\] 
over $C^\circ: = C\setminus 0$ that does not extend to an isomorphism $(X,\Delta) \simeq (X',\Delta')$ over $C$. 
Furthermore, assume $C$ is affine and there exists $t \in \cO(C)$ so that ${\rm div}_{C}(t) = 0 $. 

 From this setup, we will construct filtrations on the section rings of the special fibers. 
Set 
$L := -r(K_X+\Delta)$ and $L':=-r(K_{X'}+\Delta')$,
where $r$ is a positive integer so that $L$ and $L'$ are Cartier. 
For each non-negative integer $m$, set
\begin{align*}
\cR_m &:=H^0(X, \cO_X(mL))   
 & \cR'_{m} &:=  H^0(X', \cO_{X'}(mL') )\\
R_m &:= H^0(X_0,  \cO_{X_0}(mL_0) )
& R'_m&:= H^0(X'_0, \cO_{X'_0}(mL'_0) ). 
\end{align*} 
Additionally, set  
$$\cR:= \displaystyle  \oplus_{ m} \cR_m, 
\quad 
R:= \displaystyle  \oplus_{ m} R_m,
\quad
\cR':= \displaystyle  \oplus_{ m} \cR'_m, 
\quad \text{ and } \quad 
R': = \displaystyle \oplus_{ m} R'_m.$$

Observe that the natural maps 
$$\cR_m \otimes k(0) \to R_m \ \ \ \mbox{ and }\ \ \ \cR'_m \otimes k(0)\to R'_m$$ 
are isomorphisms.
Indeed, Kawamata-Viehweg applied to the fibers of $\pi$ and $\pi'$ implies
 $R^{i}\pi_* \cO_{X}(mL)$ and $R^{i}\pi'_* \cO_{X'}(mL')$  vanish for all $i >0$ and $m \geq 0$. 
 Hence,  $\pi_* \cO_{X}(mL)$ and $\pi'_* \cO_{X'}(mL')$ are vector bundles and their cohomology commutes with base change.
 Since $C$ is affine, $\cR_m$ and $\cR'_m$ can be identified with the $\cO_C$-module $\pi_* \cO_{X}(mL)$ and $\pi'_* \cO_{X'}(mL')$, and the statement follows.

Fix a common log resolution $\widehat{X}$ of $(X,\Delta)$ and $(X',\Delta')$ 
\begin{center}
\begin{tikzcd}[column sep=scriptsize, row sep=scriptsize]
  & \widehat{X}  \arrow[dr, "\psi'"] \arrow[dl,swap,"\psi"]& \\
  X  \arrow[rr, dashrightarrow, "\phi"] & & X'
   \end{tikzcd} 
   \end{center}
   and write $\widetilde{X}_0$ and $\widetilde{X}'_0$ for the birational transforms of $X_0$ and $X'_0$ on $\widehat{X}$. 
 Set 
 \begin{eqnarray}\label{e-defa}
a: = A_{X,\Delta +X_0}(\widetilde{X}'_0) \mbox{\ \ \ and \ \ \ }a' := A_{X',\Delta'+ X'_0 }(\widetilde{X}_0).
\end{eqnarray}
Observe that  $\widetilde{X}_0\neq \widetilde{X}'_0$, since otherwise $\phi$ would extend to an isomorphism over $C$ by the second paragraph of the proof of Proposition \ref{p-filling}.

\subsubsection{Definition of filtrations} 
For each $p\in \Z$ and $m\in \N$, set
  \[
  \cF^p \cR_m:= \{ s\in  \cR_m \, \vert \, \ord_{\widetilde{X}'_0} (s) \geq p \}, 
  \quad \text{ and } \quad 
    \cF'^p \cR'_m := \{ s\in \cR'_m \, \vert \, \ord_{\widetilde{X}_0} (s) \geq p \}
    .\]
We define $\N$-filtrations of $R$ and $R'$ by setting 
\[
\cF^{p} R_m : = \im( \cF^p \cR_m \to R_m) 
\quad
\text{ and }
\quad
\cF'^{p} R'_m : = \im( \cF'^p \cR'_m  \to R'_m),
\]
where the previous maps are given by restriction of sections.  It is straightforward to check that $\cF$ and $\cF'$ are filtrations of $R$ and $R'$. 

Observe that $R_m \simeq \cR_m / t \cR_m$ and  $\cF^p R_m \simeq \im ( \cF^p \cR_m \to \cR_m / t \cR_m) \simeq \frac{ \cF^p \cR_m }{ \cF^p \cR_m \cap t \cR_m}$ and similar statement holds for $\cF'$. 
Therefore, we have natural isomorphisms
\begin{align}\label{e-grformula}
\gr_{\cF}^p R_m \simeq \frac{ \cF^p \cR_m}{  (\cF^p \cR_m \cap t \cR_m) + \cF^{p+1} \cR_m} \quad\,
\gr_{\cF'}^p R'_m \simeq \frac{ \cF'^p \cR'_m}{  (\cF'^p \cR'_m \cap t \cR'_m) + \cF'^{p+1} \cR'_m}. 
\end{align}

\subsubsection{Relating the filtrations}
We aim to show that $\gr_{\cF} R$ and $\gr_{\cF'} R'$ are isomorphic up to a grading shrift.

 Since $\psi^*(X_0) ={ \psi'}^*(X'_0)$ have multiplicity one along $\widetilde{X}_0$ and $\widetilde{X}'_0$, we may write 
\[
K_{ \widehat{X}} + \psi_{*}^{-1}(\Delta)  = \psi^*(K_X+\Delta)  + a \widetilde{X}'_0 + P 
\]
and 
\[
K_{ \widehat{X}} + {\psi'_{*}}^{-1}(\Delta')  = \psi'^*(K_{X'}+\Delta')  + a' \widetilde{X}_0 + P',
\]
where the components of $\Supp(P) \cup \Supp(P')$ are both $\psi$ and $\psi'$-exceptional. 
Now,
\begin{align*}
\cF^p \cR_m& \simeq H^0\Big(
\widehat{X},\cO_{\widehat{X}} \big(m \psi^* L -p \widetilde{X}'_0 \big)\Big)  \\
	&= H^0 \left(
	\widehat{X}, \cO_{\widehat{X}} \big(m \psi'^*L' +(mra- p )\widetilde{X}'_0- mra' \widetilde{X}_0+mr(P-P' ) \big) 
	\right). 
\end{align*}
Hence, for $s\in \cF^p  \cR_m$, multiplying $\psi^*s$ by $t^{mra -p}$ gives an element of 
\[
H^0 \Big(\widehat{X},\cO_{\widehat{X}}\big(m\psi'^*L'  -( mr(a+a')- p) \widetilde{X}_0\big) \Big),\]
which can be identified with  $\cF'^{mr(a+a')-p}\cR'_m$. 

As described above, for each $p\in \Z$ and $m\in \N$, there is a map
 \[\tilde{\varphi}_{p,m}:
\cF^{p}\cR_m \longrightarrow 
\cF'^{mr(a+a')-p} \cR'_m
,\]
which, when  $\cR_m$ and $\cR'_m$ are viewed as submodules of $K(X)$ and $K(X')$,  sends $s \mapsto t^{mra-p}({\phi^{-1}})^*(s)$. 
 Similarly, there is a map 
 \[\tilde{\varphi}'_{p,m}:
\cF'^{p} \cR'_m \longrightarrow 
{\cF}^{mr(a+a')-p} \cR_m
,\] which
sends  $s'  \mapsto t^{mra'-p}\phi^*(s')$.
Observe that 
 $\tilde{\varphi}'_{mr(a+a')-p,m}\circ \tilde{\varphi}_{p,m} $ is the identity map, since the composition corresponds to multiplication by $t^{mra'-(mr(a+a')-p) }t^{mra-p}=1$. 
 Hence, $ \tilde{\varphi}_{p,m}$ is an isomorphism.

\begin{lem}\label{l-Fcompare}
For each $p\in \Z$ and $m \in \N$, 
\begin{itemize}
\item[(1)]   $\tilde{\varphi}_{p,m} (\cF^p \cR_m   \cap t\cR_m)= \cF'^{mr(a+a')-p+1} \cR'_m$;
\item[(2)]   $ \tilde{\varphi}_{p,m} ( \cF^{p+1} \cR_m) = \cF'^{mr(a+a') -p } \cR'_m \cap  t \cR'_m$.
\end{itemize}
\end{lem}
\begin{proof}
To see (1), fix $ s\in \cF^p \cR_m$. Now, $s\in t \cR_m$ if and only if $s$ vanishes along $X_0$, which
is equivalent to the condition that
$$
\psi^*s \in H^0\Big(\widehat{X}, \cO_{\widehat{X}} \big( m \psi^* L -p \widetilde{X}'_0 -  \widetilde{X}_0 \big) \Big)
.$$
 Since the latter holds precisely when
  \[t^{mra-p}\psi^*s
  \in H^0\Big(
  \widehat{X}, \cO_{\widehat{X}} \big(m \psi'^* L' - (mr(a+a')-p +1) \widetilde{X}_0 \big)\Big),\]
 which is identified with $\cF'^{mr(a+a')-p +1 }\cR'_m$, (1) holds.  (2) follows  from a similar argument.
 \end{proof}

 \begin{prop}\label{p-isomgrF}
The collection of maps $(\tilde{\varphi}_{p,m})$ induce an isomorphism of graded rings 
 \[
\varphi: \bigoplus_{ m\in \N  }  \bigoplus_{p\in \mathbb{Z} }{\rm gr}_\cF^p R_m 
\to
 \bigoplus_{m \in \N } \bigoplus_{p \in \mathbb{Z}} {\rm gr}_{\cF'}^{mr(a+a')-p}R'_m . \]
Hence,
 $ {\rm gr}_\cF^p R_m$ and ${\rm gr}_{\cF'}^p R'_m$ vanish for $p>mr(a+a')$. 
 \end{prop}

   \begin{proof}
For fixed $p\in \Z$ and $m \in \N$, consider the natural maps
   \begin{align*}
\varrho: \cF^p \cR_m \to \gr_{\cF}^p R_m   \quad  { \text{ and } }  \quad
\varrho':\cF'^{mr(a+a')-p}  \cR'_m \to \gr_{\cF'}^{mr(a+a')-p} R'_m 
   .\end{align*}
 By    \eqref{e-grformula} 
and  Lemma \ref{l-Fcompare}, $\tilde{\varphi}_{p,m}$ 
sends the kernel of $ \varrho$ 
to the kernel $\varrho'$. 
Hence, $\tilde{\varphi}_{p,m}$ induces an isomorphism 
$ \varphi_{p,m}: \gr_{\cF}^p R_m \to \gr_{\cF'}^{mr(a+a')-p} R'_m$.

Write $\varphi: \gr_{\cF} \cR \to \gr_{\cF'} R'_m$ for the induced module isomorphism on the direct sums. 
Since
$$\tilde{\varphi}_{p,m}( \tilde{s}_1 )  \tilde{\varphi}_{q,\ell }( \tilde{s}_2)  = \tilde{\varphi}_{p+q,m+\ell}(\tilde{s}_1  \cdot \tilde{s}_2) $$
for  any $\tilde{s}_1\in \cF^{p}\cR_{m}$ and $\tilde{s}_2 \in \cF^{q} \cR_{\ell}$, 
we see  $\varphi$ is a ring  isomorphism.
Since $ {\rm gr}_\cF^p R_m$ and ${\rm gr}_{\cF'}^pR'_m$ vanish when $p<0$,  the isomorphism $\varphi$ implies the vanishing  when $p>mr(a+a')$. 
 \end{proof}


\subsubsection{Properties of the filtrations}

\begin{lem}\label{l-filtideals}
For each positive integer $p$,
\[
\fb_{p}(\cF ) = \fa_p(\ord_{\widetilde{X}'_0}) \cdot \cO_{X_0}
\quad \text{ and } \quad 
 \fb_{p}(\cF') = \fa_p(\ord_{\widetilde{X}_0}) \cdot \cO_{X'_0}.\]
\end{lem}

\begin{proof}
Recall that
$$\fb_{p}(\cF) :=\im ( \cF^p R_m \otimes \cO_{X_0}(-mL_0) \to \cO_{X_0}) $$
for  $m \gg0$.
Since $\cF^p R_m : = \im (\cF^p \cR_m  \to R_m)$, we see
\[
\fb_p(\cF)= \im \big( \cF^p \cR_m \otimes \cO_{X}(-mL) \to \cO_{X}\big) \cdot \cO_{X_0}.\] 
Therefore, proving the first equality reduces to showing 
\begin{equation*}
\fa_{p}(\ord_{\widetilde{X}'_0})
= 
 \im \big( \cF^p \cR_m \otimes \cO_{X}(-mL)  \to \cO_{X}\big)
 \end{equation*} 
 for  $m \gg0$. Since $\cF^p \cR_m = H^0(X, \cO_{X}(mL) \otimes \fa_{p}(
 \ord_{\widetilde{X}'_0}))$ and $L$ is $\pi$-ample, the latter statement holds. The argument for $\fb_p(\cF')$ is the same. \end{proof}
 
 \begin{prop}\label{p-lctF}
The following hold:
\begin{enumerate}
\item $
 a \, \geq \lct(X\,,\Delta+ X_0 \, ; \fa_\bullet(\ord_{\widetilde{X}'_0})) = \lct(X_0,\Delta_0; \fb_\bullet( \cF ))$;
  \item 
 $a' \geq \lct(X',\Delta'+ X'_0 ; \fa_\bullet(\ord_{\widetilde{X}_0}) )= \lct(X'_0,\Delta'_0; \fb_\bullet(\cF' ))$.
 \end{enumerate}
 \end{prop}
 
 \begin{proof} The first pair  of inequalities holds by \eqref{e-lctA}. The second pair follows  from Lemma \ref{l-filtideals} and inversion of adjunction. \end{proof}
 
 \begin{prop}\label{p-TSa}
 The filtrations $\cF$ and $\cF'$ of $R$ and $R'$ are linearly bounded, non-trivial, and satisfy
$$ a+a' = S(\cF) + S(\cF').$$
\end{prop}

\begin{proof}
 Proposition \ref{p-isomgrF} implies
 $\cF^{p} R_m=0$ and $\cF'^{p}R'_m=0$ when $m>0$  and $p > mr(a+a') $. 
 Therefore, $\cF$ and $\cF'$ are linearly bounded.
 
The base ideals $\fb_p(\cF)$ and $\fb_{p'}(\cF')$ are non-zero for $p>0$ by Lemma \ref{l-filtideals}.
Therefore, the filtrations cannot be  trivial.
 
Applying Proposition \ref{p-isomgrF}, we   see
\begin{align*}
\sum_{p \geq 0 } \left( p \dim \gr_{\cF}^p R_m  \right) 
+ 
\sum_{p \geq 0 } \left( p \dim \gr_{\cF'}^p R'_m  \right) 
&= 
\sum_{p \geq 0 } \left( p \dim \gr_{\cF}^p R_m  \right) 
+ 
\sum_{p \geq 0 } \big( p \dim \gr_{\cF}^{mr(a+a')-p} R_m  \big) 
\\
&=\sum_{p \geq 0 } \left( mr(a+a')  \dim \gr_{\cF}^p R_m  \right). \\
& = mr(a+a') \dim R_m .
\end{align*}
Combining the previous equation with \eqref{e-Riemsum} gives  $S(\cF) + S(\cF') = a+a'$. 
  \end{proof}

It also natural to rescale the above values and set
\[
\beta : = (-K_{X_0}-\Delta_0)^n ( a - S(\cF) ) 
\quad \text{ and } \quad
\beta' : = (-K_{X'_0}-\Delta'_0)^n ( a' - S(\cF') ) .  \]
In this language, Proposition \ref{p-TSa} states that $\beta + \beta' = 0$.

\subsection{Proof of Theorem \ref{t-main}}\label{ss-twodeg}

The goal of this subsection is to prove Theorem \ref{t-main}. 
To do so, we consider the filtrations  defined in Section \ref{ss-filt}. 
Under the hypothesis that $(X_0,\Delta_0)$ and $(X'_0,\Delta'_0)$ are K-semistable, 
we will show that the filtrations are induced by dreamy divisors.  

Furthermore, we will prove that these dreamy divisors induce special test configurations $(\cX,\cD)$ and $(\cX',\cD')$ of $(X_0,\Delta_0)$ and $(X'_0,\Delta'_0)$ with generalized Futaki invariant zero. Hence, the log Fano pairs cannot be K-stable.
Proposition \ref{p-isomgrF}  will then be used to show that $(\cX_0, \cD_0) \simeq (\cX'_0, \cD'_0)$ and allow us to conclude that $(X_0,\Delta_0)$ and $(X'_0,\Delta'_0)$ degenerate to a common K-semistable log Fano pair.

\begin{proof}[Proof of Theorem \ref{t-main}]
Assume $(X_0,\Delta_0)$ and $(X'_0,\Delta'_0)$ are both K-semistable and  $\phi$ does not extend to an isomorphism. 
We must show $(X_0,\Delta_0)$ and $(X'_0,\Delta'_0)$ are S-equivalent and not K-stable. 
To do so, we use the filtrations $\cF$ and $\cF'$ constructed in Section \ref{ss-filt}. 

Since $(X_0,\Delta_0)$ and $(X'_0,\Delta'_0)$ are K-semistable, 
Proposition \ref{p-deltafilt} implies 
$$
\lct( X_0, \Delta_0; \fb_{\bullet}(\cF))  \geq S(\cF)
\quad \text{ and } \quad
\lct( X'_0, \Delta'_0; \fb_{\bullet}(\cF')) \geq S(\cF').
$$
Combining the previous inequalities with Propositions \ref{p-lctF} and \ref{p-TSa}, we see 
 \begin{equation}\label{e-alS1}
 a = \lct(X,\Delta+ X_0 ; \fa_\bullet(\ord_{\widetilde{X}'_0}) )
 = \lct(X_0,\Delta_0; \fb_\bullet( \cF)) = S(\cF)
 \end{equation}
 and 
 \begin{equation}\label{e-alS2}
 a' = \lct(X',\Delta'+ X'_0 ; \fa_\bullet(\ord_{\widetilde{X}_0}) )
 = \lct(X'_0,\Delta'_0; \fb_\bullet( \cF')) = S(\cF')
 .\end{equation}
 Furthermore, $\delta(X_0,\Delta_0) =\delta(X'_0,\Delta'_0)=1$.

By the first pair of equalities in  \eqref{e-alS1} and \eqref{e-alS2}, 
we may apply Proposition \ref{p-extract} to extract $\widetilde{X}'_0$ over $X$ and $\widetilde{X}_0$ over $X'$. 
Specifically, there exist proper birational morphisms $\mu$ and $\mu'$:
 \begin{center}
\begin{tikzcd}[column sep=small]
 \mathllap{V \cup W \subset{}} Y \arrow[d,swap, "\mu"] & & Y' \mathrlap{{}\supset V' \cup W'}  \arrow[d, "\mu'"] \\
 \mathllap{X_0 \subset{}} X  \arrow[rr, dashrightarrow, "\phi"] \arrow[dr, swap,"\pi"] &  & X' \mathrlap{{}\supset X'_0} \arrow[dl, "\pi'"] \\
   & C &
   \end{tikzcd}
\end{center}
such that the following hold:
 \begin{enumerate}
\item  the fibers of $Y$ (respectively, $Y'$) over $0$ contains two components $V$ and $W$ (respectively, $V'$ and $W'$) and they are the birational transforms of $X_0$ and $X'_0$;
\item  $-W$ and $-V'$ are ample over $X$ and $X'$ respectively;
\item  $(Y,V+W +\mu^{-1}_* \Delta) $ and $(Y',V'+W'+\mu'^{-1}_* \Delta')$ are lc.
\end{enumerate}
\medskip
We write
\[
\mu_0: V\to X_0\quad \text{ and } \quad
\mu'_0: W' \to X'_0
\]
for the restrictions of $\mu$ and $\mu'$ to $V$ and $W'$.
Clearly, $\mu_0$ and $\mu'_0$ are proper birational morphisms.

\begin{lem}\label{l-normal}
The pairs $(Y,V+\mu^{-1}_* \Delta)$  and $(Y',W'+\mu'^{-1}_* \Delta')$ are plt. 
Hence, $V$ and $W'$ are normal. 
\end{lem}

\begin{proof}
By inversion of adjunction, $(X,X_0 +\Delta)$ is plt. 
Therefore, $(Y,V+\mu^{-1}_* \Delta)$ is plt away from  ${\rm Exc}(\mu)=W$.
Since  $(Y,V+W+\mu^{-1}_* \Delta)$ is lc,  $(Y,V+\mu^{-1}_* \Delta)$ cannot have lc centers in $W$. 
Therefore, $(Y,V+\mu^{-1}_* \Delta)$ is plt, and  $V$ is normal  by \cite[5.52]{KM98}.
The same argument works for $Y'$. 
\end{proof}

Now, consider the restrictions of $W$ and $V'$ to the birational transforms of $X_0$ and $X'_0$:
 $$E:= W \vert_{V}
 \quad 
 \text{ and }
 \quad
 E': = V'\vert_{ W'}.$$
Since $W$ and $V'$ are $\Q$-Cartier, but not necessarily Cartier, $E$ and $E'$ may have fractional coefficients.

The $\Q$-divisors $E$ and $E'$ induce $\N$-filtrations on $R$ and $R'$ defined by  
\[
\cF_{E}^p R_m := H^0\Big( V, \cO_V \big(  \mu_0^*(mL_0) -  \lceil p E \rceil \big) \Big) \subseteq R_m
\]
and
\[
\cF_{E'}^p R'_m := H^0\Big( W',  \cO_{W'} \big( {{\mu_0}'}^*(mL'_0) - \lceil  p E' \rceil \big) \Big) \subseteq R'_m
\]
for $p,m \geq 0$.
Note that  
$$\cF^p R_m \subseteq \cF_E^p R_m 
\quad \text{ and } \quad 
\cF'^p R'_m \subseteq \cF_{E'}^p R'_m.$$
Therefore,
\begin{equation}
\label{e-SFF_E}
S(\cF) \leq S(\cF_E)
\quad\text{ and } \quad
S(\cF') \leq S(\cF_{E'}).
\end{equation}

\begin{lem}\label{l-SuppE}
The supports $F: = \Supp(E)$ and $F':= \Supp(E')$ are prime divisors
Furthermore, 
\begin{itemize}
\item[(1)] $F$ computes $\d(X_0,\Delta_0)$ and $E= \frac{1}{d} F$ for some positive integer $d$;
\item[(2)] $F'$ computes $\d(X'_0,\Delta'_0)$ and $E'= \frac{1}{d'} F'$ for some positive integer $d'$.
\end{itemize}
\end{lem}

\begin{proof} 
Since $-W$ is ample over $X$, the restriction map 
\[ 
\mu_* \cO_Y(-pW) \to   {\mu_0}_* \cO_V(-pE)  
\]
is surjective for all positive integers $p$ sufficiently divisible. 
Hence, if we set $$\fa_p(E): = {\mu_{0}}_* \cO_V(-pE) \subseteq \cO_{X_0},$$ then $\fa_p(E) = \fa_p(\ord_W) \cdot \cO_{X_0}$ for such $p$ and inversion of adjunction yields
\[
\lct(X_0,\Delta_0; \fa_\bullet(E)) 
 = \lct(X,\Delta +X_0; \fa_\bullet( \ord_W) ).\]
Combining the previous equality with  \eqref{e-alS1} and \eqref{e-SFF_E}
yields 
\begin{equation}\label{e-lctE<S}
\lct(X_0,\Delta_0; \fa_\bullet(E)) = S(\cF) \leq S(\cF_E).
\end{equation}
Since $(X_0,\Delta_0)$ is K-semistable, 
Proposition \ref{p-Qdivcomputing} implies \eqref{e-lctE<S} is an equality and $F: = \Supp(E)$ is a prime divisor. 
Therefore, $S(\cF) =S(\cF_E)$ and  $F:=\Supp(E)$ is a prime divisor that computes $\d(X_0,\Delta_0)$.

To see $E= \frac{1}{d} F$ for a positive integer $d$, we cut by hyperplanes to reduce the statement to a surface computation. The statement then follows from the fact that 
  $(Y,V+W+\mu_*^{-1}(\Delta))$ is lc and  the classification of lc surface pairs (\cite[3.32]{Kol13} and  \cite[3.35.2]{Kol13}). The argument for $E'$ is identical.
\end{proof}

\begin{lem}\label{l-volF}
For all but finitely many $x\in \R_{\geq 0}$, 
$$\vol(\cF_E R^{(x)}) =\vol(\cF R^{(x)})
\quad \text{ and } \quad
\vol(\cF_{E'} R'^{(x)}) =  
\vol(\cF' R'^{(x)}).$$
\end{lem}

\begin{proof}
As shown in the proof of Lemma \ref{l-SuppE}, $S(\cF)= S(\cF_E)$. 
Hence, 
$$ \frac{1}{r L^n } \int_0^\infty \vol(\cF R^{(x)})\, dx  = S(\cF) = S(\cF_E)=   \frac{1}{r L^n } \int_0^\infty \vol(\cF_E R^{(x)})\, dx.$$
Since  $\vol(\cF R^{(x)}) \leq  \vol(\cF_E R^{(x)})$  by \eqref{e-SFF_E} and the two functions 
are continuous at all but one value  \cite[5.3.ii]{BHJ17}, the desired equality holds. \end{proof}

\begin{prop}\label{p-F=F_E}
For all $p\in \Z$ and $m\in\N$, 
\[\cF^p R_m=\cF_{E}^p R_m\quad \text{ and  } \quad
\cF'^p R'_m=\cF_{E'}^p R'_m.\]
\end{prop}

Proving this key proposition amounts to showing that the restriction map 
\[
H^0\Big( Y, \cO_Y \big( m \mu^*L-pW \big)\Big) \to H^0 \Big(V, \cO_V \big( m \mu_0^{*}L_0- \lceil  p E \rceil \big)\Big) 
\]
is surjective. 
Since such a statement is quite subtle, we will not study this restriction map directly. 
Instead, we use a construction that originated in \cite{Li17} (with a refining analysis from \cite{LWX18})
and work on the cone over our family of log Fano pairs. 

\medskip
Consider the relative cone over $(X,\Delta) \to C$ with polarization $L$ given by  
$$Z : = C(X/C, L)=  \Spec (\cR)  \to C.$$ Write $\sigma: C \to  Z$ for the section of cone points and $\Gamma$ for the closure of the inverse image of $\Delta$ under the projection $Z \setminus \sigma(C)  \to X$. 
Note that the fiber of $(Z,\Gamma)$ over $0$, denoted $(Z_0,\Gamma_0)$, is the  cone over $(X_0,\Delta_0)$ and $Z_0 =\Spec( R)$. 

There is a natural  proper birational morphism 
$Y_{L}\to Z$, where $Y_{L } : = \Spec_{Y} \big( \bigoplus_{m\geq 0} \cO_{Y}(mL)\big)$,  and it is the total space of the line bundle whose sheaf of sections is $\cO_{Y}(mL)$. 
We write $Y_{\rm zs} \subset Y_L$ for the zero section  and $W_\infty$ for the preimage of $W$ under the projection map $Y_{L }\to Y$. Hence, $Y_{\rm zs} \cap W_\infty \simeq W$. 

Associated to the divisor $W$ over $X$, is a ray of valuations
$$\{ w_t \, \vert \, t\in [0,\infty ) \} \subset  \Val_{Z },$$
where $w_t$ is the quasi-monomial valuation  with weights $(1,t)$ along $Y_{\rm zs}$ and $W_\infty$. 
For each positive integer $k$, let  $W_k$ denote the divisor over $X$ such that $w_{1/k} = \frac{1}{k} \ord_{W_k}$.

Note that 
\begin{eqnarray*}
A_{Z,\Gamma +Z_0}(w_t) = A_{Z,\Gamma +Z_0}( \ord_{Y_{\rm zs}} )+ t A_{Z,\Gamma +Z_0}(\ord_{W_\infty}) 
&=&A_{Z,\Gamma +Z_0}( \ord_{Y_{\rm zs}} )+ t A_{X,\Delta+X_0}(\ord_W) \\
&=&r^{-1} + t a 
\end{eqnarray*}
and, by a local computation as in \ref{ss-nvolkss}, 
\begin{equation}\label{e-a(w_t)}
\fa_{p}(w_t)= \bigoplus_{m \in \N} \cF^{(p-m)/t} \cR_m \subseteq \cR.
\end{equation}
Therefore, 
\begin{equation}\label{e-a(w_t).}
\fa_{p}(w_t) \cdot \cO_{Z_0} =  \bigoplus_{m \in \N} \cF^{(p-m)/t} R_m \subseteq R.
\end{equation}

We also consider a ray in the valuation space of $Z_0$.
Consider the natural map $V_{L_0}  \to Z_0$, where $V_{L_0} = \Spec_{V} \big( \bigoplus_{m \geq 0} \cO_V(m\mu_0^*L_0) \big)$.
 We write $V_{\rm zs} \subset V_{L_0} $ for the zero section and $F_{\infty}$ for the inverse image of $F$ under the projection $V_{L_0} \to V$. 
 Let $v_t$ denote the  the quasi-monomial valuation with weights $(1,td)$ along $V_{\rm zs}$ and $F_{\infty}$. 
Note that
\begin{equation}\label{e-a(v_t)}
\fa_{p}(v_t) = \bigoplus_{m \in \N} \cF_{F}^{(p-m)d/t } R_m= \bigoplus_{m \in \N} \cF_{E}^{(p-m)/t } R_m\subseteq R. 
\end{equation}
For  each positive integer $k$ divisible by $d$, there is a divisor $F_k$ over $Z_0$ such that $v_{1/k} = \frac{d}{k} \ord_{F_k}$.
Since $F$ computes $\d(X_0,\Delta_0)=1$, 
the proof of  Theorem \ref{t-delta=1} implies that $F_k$ may be extracted for $k\gg0$. 
Let 
$$ 
\rho_k: Z_{0,F_k} \to Z_0
$$
denote this extraction.
 
 \begin{lem}\label{l-Wkextract}
 For $k\gg 0$, there exists an extraction $\tau_k : Z_{W_k} \to Z$ of $W_k$ over $Z$ such that
 $(Z_{W_k},W_k + {{\tau_k}_*}^{-1}(\Gamma + Z_0))$ is lc. 
  \end{lem}
  
\begin{proof}
For $t\in \R_{\geq 0}$, let $\fa_\bullet(w_t) \cdot \cO_{Z_0}$ denote the restriction of $\fa_\bullet(w_t)$ to a graded sequence of ideals on $Z_0$.
By Lemma  \ref{l-dervolume}.1 and Lemma \ref{l-volF}, 
\begin{equation}\label{e-multbaw}
\mult(\fa_\bullet(w_{t})\cdot \cO_{Z_0}) = \mult(\fa_{\bullet}(v_t)) 
\end{equation}
for each $t\in \R_{\geq 0}$.

Set $$
f(t) := \left( r^{-1}+at\right)^{n+1} \mult(\fa_{\bullet}(w_t) \cdot \cO_{Z_0}).$$
Applying Lemma \ref{l-dervolume}, we see 
\[
f(0) = (-K_{X_0}-\Delta_0)^n
 \quad \text{ and } \quad
f'(0) = (-K_{X_0}-\Delta_0)^n ( a - S(\cF))=0.\]
Hence, a Taylor expansion gives $f(t) = f(0) + O(t^2)$ for $0 < t \ll 1$.

For each positive integer $k$, define
\[
c_k := \lct (Z, \Gamma + Z_0; \fa_\bullet(\ord_{W_k} ) ) .\]  
Note that
$
c_k \leq 
A_{Z, \Gamma +Z_0}( W_k) = kr^{-1}+a$
by  \eqref{e-lctA}.
Additionally,
\[
c_k
= k \cdot  \lct(Z, \Gamma+Z_0; \fa_{\bullet}(w_{1/k})) 
= k \cdot  \lct(Z_0, \Gamma_0; \fa_{\bullet}(w_{1/k}) \cdot \cO_{Z_0})
\]
by inversion of adjunction and the relation $\fa_{\bullet k}(  \ord_{W_k} ) = \fa_{\bullet } (w_{1/k})$. 
Therefore, 
  \begin{multline*}
f(0) \leq \lct\big(Z_0,\Gamma_0;\fa_{\bullet}(w_{1/k}) \cdot \cO_{Z_0} \big)^{n+1} \mult \left(\fa_{\bullet}(w_{1/k}) \cdot \cO_{Z_0} \right) \\
 = \left(\frac{c_k}{k}\right)^{n+1}  \mult \left( \fa_{\bullet} (w_{1/k}) \cdot  \cO_{Z_0} \right) 
 \leq  f\bigg(\frac{1}{k}\bigg), 
  \end{multline*}
where the first inequality follows from \cite[7]{Liu18} and the assumption that $(X_0,\Delta_0)$ is K-semistable (see also \cite[Theorem 3.1]{Li17} and \cite[Theorem A]{LX16}). 

Now, set $a_k:=  A_{Z,\Gamma+ Z_0}( W_k) - c_k$. As in the proof of Theorem \ref{t-delta=1}, 
the previous inequalities imply $\displaystyle \lim_{k \to \infty} a_k = 0$. 
Hence, if $k\gg0$, Proposition \ref{p-extract} yields an extraction $\tau_k : Z_{W_k} \to Z$ of $W_k$ such that the pair
$$(Z_{W_k}, {\tau_{k}}^{-1}_* ( \Gamma +Z_0) + (1-a_k) W_k) $$
 is lc. 
 Since $\displaystyle \lim_{k \to \infty} (1-a_k) =1$, the ACC for log canonical thresholds \cite{HMX14} implies $(Z_{W_k}, {\tau_{k}}^{-1}_* (\Gamma+Z_0) +  W_k)$ must be lc when $k\gg0$. 
 \end{proof}
 
 From now on, we fix a positive integer $k$ so that $d$ divides $k$, there exist extractions
 $$\rho_k : Z_{0,F_k} \to Z_0
\quad 
\text{ and } 
\quad
\tau_k : Z_{W_k} \to Z,
$$
and $(Z_{W_k}, {\tau_{k}}^{-1}_* (\Gamma+Z_0) +  W_k)$ is lc.
The argument used to prove Lemma \ref{l-normal} implies $(Z_{W_k}, {\tau_{k}}^{-1}_* (\Gamma+Z_0))$ is plt and ${\tau_{k}}^{-1}_*(Z_0)$ is normal.

 \begin{lem}\label{l-diagram}
 We have a diagram 
  \begin{center}
{\begin{tikzcd}
 \mathllap{ F_k \subset\, } Z_{0,F_k} \arrow[d,"\rho_k"] \arrow[r, hook] & \mathrlap{ Z_{W_k} \,\supset W_k} \arrow[d,"\tau_k"] \\
Z_0 \arrow[r, hook] & Z 
\end{tikzcd} }
\end{center}
(i.e. the birational transform of $Z_0$ on $Z_{W_k}$ is the extraction of $F_k$).
Additionally,
\begin{enumerate}[(i)]
\item $W_k\vert_{Z_{0,F_k}} = \frac{1}{d} F_k$ and
\item $dW_k$ is Cartier at the generic point of $F_k$. 
\end{enumerate}
\end{lem} 

\begin{proof}
Since $-W_k$ and $-F_k$ are ample over $Z$  and $Z_0$, 
we may find a positive integer $p$ so that
$$\fa_{pd}(\ord_{W_k}) \subseteq \cO_{Z}
\quad \text{ and } \quad
\fa_{p}(\ord_{F_k}) \subseteq \cO_{Z_0}$$
satisfy the conclusions of Lemma \ref{l-antiample}.
Hence, $Z_{W_k}$ is the blowup of $Z$ along $\fa_{pd}(\ord_{W_k})$ and $Z_{0,F_k}$ is the blowup of $Z_0$ along $\fa_{p}(\ord_{F_k})$.
The former statement implies  ${\tau_{k}}_{*}^{-1}(Z_0)$ is the blowup of $Z_0$ along $\fa_{pd}(\ord_{W_k})\cdot \cO_{Z_0}$. 

\medskip

\noindent \emph{Claim:} $\mult(\fa_{pd}(\ord_{W_k})\cdot \cO_{Z_0}) = \mult(\fa_{p}(\ord_{F_k}))$

To compute these multiplicities, observe 
\[
\mult  (\fa_{p}(\ord_{F_k}) )  = p^{n+1} \cdot  \mult  \left(\fa_\bullet (\ord_{F_k})\right)
= (pd/k)^{n+1} \mult \left(\fa_{\bullet}(v_{1/k}) \right)
,\]
since $\fa_{p \ell}(\ord_{F_k}) = \fa_{p}(\ord_{F_k})^\ell$ for all $\ell \geq0$  
and $\frac{d}{k} \ord_{F_k} = v_{1/k}$. 
Similar reasoning implies
\[
\mult  (\fa_{pd}(\ord_{W_k}) \cdot \cO_{Z_0} ) 
= (pd/k)^{n+1} \mult \left(\fa_{\bullet}(w_{1/k}) \cdot \cO_{Z_0} \right). 
\]
Equation \ref{e-multbaw} now completes the claim.
\medskip
 
Observe $\fa_{pd}(\ord_{W_k}) \cdot \cO_{Z_0} \subseteq \fa_p(\ord_{F_k})$,
since
$$\fa_{pd}(\ord_{W_k}) \cdot \cO_{Z_0} = \bigoplus_{m \in \N} \cF^{pd-mk} R_m \subseteq  \bigoplus_{m \in \N} \cF^{pd-mk}_{E} R_m  = \fa_{p} ( \ord_{F_k}) .$$ 
 A theorem of Rees \cite{Rees61} now implies that $\fa_{pd}(\ord_{W_k})\cdot \cO_{Z_0}$ and $\fa_{p}(\ord_{F_k})$   have the same integral closure.
  
 The latter implies $\fa_{pd}(\ord_{W_k})\cdot \cO_{Z_0}$ and $\fa_{p}(\ord_{F_k})$ have the same normalized blowups.
 Since the corresponding blowups equal ${\tau_{k}}_{*}^{-1}(Z_0)$ and $Z_{0, F_k}$ and are already normal, they must be isomorphic.
The equality of the integral closures further implies  $pdW_k\vert_{{\tau_{k}}_{*}^{-1}(Z_0)} = p F_k$, which completes  (1). 

To see (2), cut by $n-1$ generic hyperplanes to get a lc surface pair. The statement then follows from the fact that 
  $(Y,V+W+\mu_*^{-1}(\Delta))$ is lc and  the classification of lc surface singularities (see \cite[3.32]{Kol13} and  \cite[3.35.2]{Kol13}).
\end{proof}

\begin{proof}[Proof of Proposition \ref{p-F=F_E}]
With the above results, the equality of the two filtrations is now a statement concerning valuation ideals (see Equations \ref{e-a(w_t).} and \ref{e-a(v_t)}).

Let us consider the restriction sequence 
\begin{equation}\label{e-sesWF}
0\to \cO_{Z_{W_k}}(-pdW_k -Z_{0,F_k}) \to \cO_{Z_{W_k}}(-pdW_k) \to  \cO_{ Z_{0,F_k}}(-pF_k) \to 0.
\end{equation}
where $p$ is a positive integer.
By  the proof \cite[5.26]{KM98}, the sequence is exact if $pdW_k$ is Cartier at all codimension two points of
$Z_{W_k}$ contained in $Z_{0,F_k}$. 
Since the latter holds  by Lemma \ref{l-diagram}, \eqref{e-sesWF}
is exact.\medskip

\noindent \emph{Claim:} $R^1 {\tau_k}_* \cO_{Z_{W_k}}( -pd W_k - Z_{0,F_k} ) =0$ for all $p>0$. 

Note that $(Z_{W_k},  {\tau_k}_{*}^{-1}( \Gamma) )$ is klt, 
since $(Z_{W_k},  {\tau_k}_{*}^{-1}( \Gamma +Z_0))$ is plt.
Therefore, \cite[10.37]{Kol13} implies the desired vanishing holds as long as
\begin{equation}\label{eq-provenef}
 -pd W_k - Z_{0,F_k} 
 - 
(
 K_{Z_{W_k}} +  {\tau_k}_{*}^{-1}( \Gamma) 
)
\end{equation}
is $\tau_k$-nef. 
To prove the latter, observe 
\begin{align*}
K_{Z_{W_k}} +  {\tau_k}_{*}^{-1}( \Gamma)  
&\sim_{\Q, \tau_k} K_{Z_{W_k}} +  {\tau_k}_{*}^{-1}( \Gamma) - \tau_k^{*} ( K_Z+ \Gamma +Z_0) \\
& =  (A_{Z,\Gamma+Z_0}( \ord_{W_k}) - 1) W_k - Z_{0,F_k}.
\end{align*}
Therefore, \eqref{eq-provenef} is relatively $\Q$-linearly equivalent to $-(pd +A_{Z,\Gamma +Z_0}(\ord_{W_k}) -1) W_k$.
Since  $-W_k$ is $\tau_k$-ample, \eqref{eq-provenef}  is $\tau_k$-nef when $p> 0$  and the proof of the claim is complete. 
\medskip

Returning to the proof of the proposition, we apply ${\tau_k}_*$ to \eqref{e-sesWF} and see 
 \[
0\to { \tau_k}_* \cO_{Z_{W_k}}(-pdW_k -Z_{0,F_k}) 
\to 
   \fa_{pd}(\ord_{W_k} ) \to \fa_{p}(\ord_{F_k}) \to 0
\]
is exact for all $p>0$. 
The  right exactness implies $\cF^{pd-mk}R_m = \cF_{E}^{pd-mk} R_m$ for all  $p >0$ and $m \geq 0$.
 Since $k$ was chosen to be a multiple of $d$, the latter implies
 $\cF^{pd}R_m = \cF_{E}^{pd} R_m$ for all $p ,m\geq 0$.
 Using the relations
 \[
 \cF^p R_m \subseteq \cF_E^p R_m = \cF_E^{ \lceil p/d \rceil d } R_m,
 \]
 we   conclude  
 $\cF^{p}R_m = \cF_{E}^{p} R_m$ for all $p,m\geq0$.
 \end{proof}
 
We now return to the proof of Theorem \ref{t-main}. 
Recall, $F := \Supp(E)$ and ${F':= \Supp(E')}$  compute $\d(X_0,\Delta_0)$ and $\delta(X'_0,\Delta'_0)$, which are both one.  
Theorem \ref{t-delta=1}  implies $F$ and $F'$ are dreamy.
Therefore, $(X_0,\Delta_0)$ and $(X'_0,\Delta'_0)$ are not K-stable.

It remains to show that $(X_0,\Delta_0)$ and $(X'_0,\Delta'_0)$ are S-equivalent.
Consider the filtrations $\cF$ and $\cF'$, 
which agree with $\cF_{E}$ and $\cF_{E'}$ by Proposition \ref{p-F=F_E}.
The filtrations $\cF$ and $\cF'$ are finitely generated (since $F$ and $F'$ are dreamy). Let $(\cX, \cD)$ and $(\cX',\cD')$ denote the test configuration of $(X_0,\Delta_0)$ and $(X'_0,\Delta'_0)$ associated to these filtrations.

We claim that $(\cX, \cD)$ and $(\cX',\cD')$ are non-trivial special test configurations and the fibers over $0 \in \mathbb{A}^1$ are K-semistable. 
Indeed, $(\cX,\cD)$ is  a normal non-trivial test configuration  \cite[3.8]{Fuj17b} and its Futaki invariant
 is a multiple of $A_{X_0,\Delta_0}(F) - S(F)$  \cite[6.12]{Fuj16}, which is zero. 
 Therefore, $(\cX,\cD)$ must be special, 
 since otherwise there would exist a test configuration of $(X_0,\Delta_0)$
  with negative Futaki invariant \cite[1]{LX14}. 
   \cite[3.1]{LWX18} now  implies  $(\cX_0,\cD_0)$ is K-semistable. Since the same argument may be applied to $(\cX',\cD')$, the claim holds. 

To finish the proof  of the S-equivalence, we are left to show that there is an isomorphism $(\cX_0,\cD_0) \simeq (\cX'_0,\cD'_0)$.  
Note that
\[
\cX_0 = \Proj  \bigg( \bigoplus_{m \in \N}     \bigoplus_{p \in \N}   \gr_{\cF}^p R_m \bigg)
\quad 
\text{ and }
\cX'_0 = \Proj \bigg( \bigoplus_{m \in \N}    \bigoplus_{p \in \N}   \gr_{\cF'}^{p} R'_m  \bigg) .\]
Therefore, the isomorphism $\varphi :\gr_{\cF} R \to  \gr_{\cF'} R'$ in Proposition \ref{p-isomgrF} induces an isomorphism $\cX_0 \simeq \cX'_0$. This proves Theorem \ref{t-main} in the case when the boundaries $\Delta$ and $\Delta'$ are trivial.
We claim that $\varphi$  indeed induces an isomorphism of 
pairs $(\cX_0,\cD_0)\simeq (\cX'_0, \cD'_0)$. Proving $\cD_0$ and $\cD'_0$ match under the isomorphism $\cX_0\simeq \cX'_0$ is quite delicate.

To proceed, fix a prime divisor $B\subset \Supp (\Delta)$, and let $B' \subset \Supp(\Delta')$ denote its birational transform on $X'$. Write $\cB \subset \Supp(\cD)$ and $\cB' \subset \Supp(\cD)$ for the degenerations of $B_0 \subset X_0$ and $B'_0\subset X'_0$ on $\cX$ and $\cX'$. To complete the proof, we will show that the isomorphism $\cX_0 \simeq \cX'_0$ sends $\cB_0$ to $\cB'_0$, where $\cB_0$ and $\cB'_0$ denote the divisorial parts of the scheme theoretic fibers of $\cB$ and $\cB'$ over $0$.

Recall that the scheme theoretic fibers of $\cB$ and $\cB'$ over $0$ are defined by the ideals
$$
\init(I_{B_0}) \subset \gr_{\cF} R
\quad
\text{ and }
\quad
\init(I_{B'_0}) \subset \gr_{\cF'} R', 
$$
where $I_{B_0} \subset R$ and $I_{B'_0} \subset R$ denote the ideals defining $B_0$ and $B'_0$.
Observe that $\init(I_{B_0})$ and $\init(I_{B'_0})$ are homogenous with respect to the gradings by $m$ and $p$. 
Furthermore, the graded components may be expressed as
\[
\init(I_{B_0})_{p,m} : = \init(I_{B_0}) \cap \gr_\cF^{p} R_m = \im  \left(  \cF^p R_m \cap I_{B_0} \to \gr_{\cF}^p R_m  \right) \simeq  \frac{  \cF^p R_m \cap  I_{B_0} }{  \cF^{p+1} R_m \cap I_{B_0} }
\]
and 
\[
\init(I_{B'_0})_{p,m} : = \init(I_{B'_0}) \cap \gr_{\cF'}^{p} R'_m = \im  \left(  \cF'^p R'_m \cap I_{B'_0} \to \gr_{\cF'}^p R'_m  \right) \simeq  \frac{  \cF'^p R'_m \cap  I_{B'_0} }{  \cF'^{p+1} R'_m \cap I_{B'_0} }
.\]

Rather than showing that the isomorphism $\gr_{\cF} R\to \gr_{\cF'} R'$ sends $\init(I_{B_0})$ 
to 
$\init(I_{B'_0})$, we introduce auxiliary ideals defined using sections of the relative section rings that vanish along $B$ and $B'$.
For $p,m\geq 0$, set
\[
I_{p,m} : = \im \left( \cF^p \cR_m \cap \cI_B \to \gr_{\cF}^p R_m \right) 
\quad 
\text{ and } 
\quad 
I'_{p,m} : = \im \left( \cF'^p \cR'_m \cap \cI_{B'} \to \gr_{\cF'}^p R'_m \right) 
,\]
where $\cI_B \subset \cR$ and $\cI_{B'}\subset \cR'$ are the ideals defining $B$ and $B'$. 
It is straightforward to check that 
\[
I := \bigoplus_{m \in \N}  \bigoplus_{p \in \N}   I_{p,m} \subset \gr_{\cF} R  
\quad
\text{ and } 
\quad 
I' := \bigoplus_{m \in \N}  \bigoplus_{p \in \N}   I'_{p,m}  \subset \gr_{\cF'} R'  
\]
are ideals and are contained in $\init(I_{B_0})$ and $\init(I_{B'_0})$ . 

The following two propositions show that the isomorphism $\cX_0 \simeq \cX'_0$ induced by $\varphi$ sends 
$\cB_0$ to $\cB'_0$.
Indeed, Proposition \ref{p-II'isom} states that the	isomorphism
$\cX_0 \simeq \cX'_0$ sends $V(I)$ to $V(I')$. 
Since $V(I)$ and $V(I')$ agree with $\cB_0$ and $\cB'_0$ away from codimension two subsets by Proposition \ref{p-initalcod2}, the result follows. 
\end{proof}

We are left to prove the following two propositions used in the above proof.

\begin{prop}\label{p-II'isom}
The isomorphism $\varphi: \gr_{\cF}R \to  \gr_{\cF'}R'$ sends $I$ to $I'$.
\end{prop}

\begin{proof}
Observe that for $\tilde{s} \in \cF^{p} \cR_m $,
\[
 \tilde{s} \in \cF^{p} \cR_m \cap \cI_B 
 \quad  \text{ if and only if }  \quad
\tilde{\varphi}_{p,m} (\tilde{s}) \in \cF'^{mr(a+a')-p} \cR'_m \cap \cI_{B'}.\]
Indeed, $\tilde{s}$ and $\tilde{\varphi}_{p,m} (\tilde{s})$ differ by a unit away from $0\in C$ and  membership in the ideals $\cI_B$ and $\cI_{B'}$ may be tested away from $0\in C$, since $B$ and $B'$ are horizontal.  Therefore, ${\varphi} (I_{p,m}) = I'_{mr(a+a')-p,m}$ and the result follows.
\end{proof}

The next proposition is more difficult to prove. 

\begin{prop}\label{p-initalcod2}
The subschemes defined by 
\begin{enumerate}
\item  $\init( I_{B_0})$ and $I$ on $\cX_0$; 
\item  $\init( I_{B'_0})$ and $I'$ on $\cX'_0$
\end{enumerate}
agree away from codimension 2 subsets.
 \end{prop}
 
To prove the statement for (1), it suffices to show that
\begin{equation}\label{e-dimn-2}
 \dim \bigg( \bigoplus_{p\geq 0} \frac{\init(I_{B_0})_{p,m} }{I_{p,m}}  \bigg) = O(m^{n-2}) 
 .\end{equation}
To bound the dimension of the previous module, 
we return to the cone construction argument used earlier in this section.
 
 Consider the the relative cone $(Z,\Gamma)$ 
 and the extractions 
 $$\tau_k:Z_{W_k} \to Z\quad \text{ and } \quad \rho_k:Z_{0,F_k}\to Z_0$$ 
 used in the proof of Proposition \ref{p-F=F_E}. 
 Let $G \subseteq \Supp( \Gamma)$ denote the prime divisor defined via pulling back $B \subseteq \Supp(\Delta)$. 
Write $\widetilde{G}$ and $\widetilde{G}_0$ for the birational transforms of $G$ and $G_0$ on $Z_{W_k}$ and $Z_{0,F_k}$.

 Observe that for $j\geq 0$
\[
\fa_{jd}(\ord_{W_k}) \cap \cI_G
=
\bigoplus_{m \geq 0 }  \left( \cF^{jd-mk} \cR_m  \cap \cI_B \right) 
\]
and
\[
 \fa_j(\ord_{F_k}) \cap \cI_{G_0}
 =
\bigoplus_{m\geq 0}  \left( \cF^{jd-mk} R_m \cap I_{B_0} \right).
\]
Therefore, 
\begin{equation}\label{e-Nideal}
\frac{
 \fa_j(\ord_{F_k}) \cap \cI_{G_0}
 }
{  
\left( \fa_{jd} (\ord_{W_k}) \cap \cI_G\right) \cdot \cO_{Z_0} 
+  \left( \fa_{j+1}(\ord_{F_k}) \cap \cI_{G_0} \right) 
}
\simeq 
\bigoplus_{m \geq 0 }   \frac{\init(I_{B_0})_{jd-mk,m} }{I_{jd-mk,m}}    
.\end{equation}

\begin{lem}\label{l-idealgrowth}
We have 
\[
\frac{
 \fa_j(\ord_{F_k}) \cap \cI_{G_0}
 }
{  
\left( \fa_{jd} (\ord_{W_k}) \cap \cI_G\right) \cdot \cO_{Z_0} 
+  \left( \fa_{j+1}(\ord_{F_k}) \cap \cI_{G_0} \right) 
}
= O(j^{n-2} ) .\]
\end{lem}

A key subtlety in proving this lemma is that the divisors $\widetilde{G}$ and $\widetilde{G}_0$
may fail to be $\Q$-Cartier. 
The proof we will utilize the fact that 
$$(Z_{W_k}, W_k + Z_{0,F_k}+ {{\tau_k}_*}^{-1} (\Gamma))$$
 is lc.
The latter implies $F_k =W_k \cap  Z_{0,F_k}$ is not
 contained in $\Supp( {{\tau_k}_*}^{-1} (\Gamma))$  by \cite[2.32.2]{Kol13}. Hence, $F_k \not\subset \tilde{G}$.

\begin{proof} 
 Fix a positive integer $q$ such that $qdW_k$ is Cartier. For each $r \in \{0, \ldots, q-1 \}$, set
 \[
 \mathcal{Q}_r:= {\rm coker} \left( 
 \cO_{Z_{W_k}}(-\widetilde{G} -rd W_k ) \to 
 \cO_{Z_{0,F_k}}(-\widetilde{G}_0 -r F_k ) 
 \right),
 \]
 where the previous map is defined via restriction.

 \medskip
 
 \noindent \emph{Claim}: 
 The support of $\mathcal{Q}_r$ is contained in the intersection of $Z_{0,F_k}$ and the locus where $\widetilde{G}$ is not $\Q$-Cartier. 
 
To prove the claim, it suffices to show 
\[
 \cO_{Z_{W_k}}(-\widetilde{G} -rd W_k ) \to 
 \cO_{Z_{0,F_k}}(-\widetilde{G}_0 -r F_k )
 \to 0 
\]
is exact along 
\[
U:= \{ z \in Z_{W_k}  \, \vert \, \widetilde{G} \text{ is $\Q$-Cartier at } z \}.\]
 The  the proof of \cite[5.26]{KM98} implies the statement holds, assuming 
\begin{itemize}
\item[(i)] $(\widetilde{G} +rd W_k)\vert_{U}$  and $Z_{0,F_k}\vert_{U}$ are $\Q$-Cartier and
\item[(ii)]$(\widetilde{G} +rd W_k)\vert_{U}$ is Cartier at all codimension two points of $U$ contained in $Z_{0,F_k} \vert_{U}$.
\end{itemize}

Statement (i) is clear, since $W_k$ and $Z_{0,F_k}$ are $\Q$-Cartier and $U$ is the locus where $\widetilde{G}$ is $\Q$-Cartier. 
For (ii), observe that $G+rdW_k$ is Cartier at the generic point of $F_k$, since $F_k \not\subset \widetilde{G}$ 
and $dW_k$ is Cartier at the generic point of $F_k$. 
 Note that $Z_{W_k}$ is regular at the remaining codimension two points contained in $Z_{0, F_k}$. Indeed, 
 $
Z_{W_k} \setminus W_k \simeq Z \setminus \sigma(C) 
$
and $Z$ is regular along all codimension one points of $Z_0$, since $Z_0$ is a normal Cartier divisor on $Z$. 
 \medskip

We now return to the proof of the lemma. 
Given a positive integer $j$, write $j=bq+r$ where $r \in \{0, \ldots, q-1 \}$.
Consider the exact sequence 
 \[
 \cO_{Z_{W_k}}(-\widetilde{G} -jd W_k ) \to 
 \cO_{Z_{0,F_k}}(-\widetilde{G}_0 -j F_k ) \to
 \mathcal{Q}_r(-bqdW_k)
 \to 0 .
 \]
Pushing forward the sequence by ${\tau_k}_*$, we see
 \[
\fa_{jd} (\ord_{W_k}) \cap \cI_G 
\to 
\fa_{j} (\ord_{F_k}) \cap \cI_{G_0} 
\to 
{\tau_k}_{*} \mathcal{Q}_r(-bqd W_k) 
\to 
0,
\]
is exact for $b\gg0$, since $-W_k$ is $\tau_k$-ample.
Hence, 
\begin{multline*}
\dim
\left( 
\frac{
 \fa_j(\ord_{F_k}) \cap \cI_{G_0}
 }
{  
\left( \fa_{jd} (\ord_{W_k}) \cap \cI_G\right) \cdot \cO_{Z_0} 
+  \left( \fa_{j+1}(\ord_{F_k}) \cap \cI_{G_0} \right) 
} 
\right) 
 \\
\leq 
\dim
\left( 
\frac{
 \fa_j(\ord_{F_k}) \cap \cI_{G_0}
 }
{  
\left( \fa_{jd} (\ord_{W_k}) \cap \cI_G\right) \cdot \cO_{Z_0} 
+  \left( \fa_{j+q}(\ord_{F_k}) \cap \cI_{G_0} \right) 
} 
\right)\\
\leq 
\dim
\big( 
{\rm coker} \left( {\tau_k}_{*} \mathcal{Q}_r(-bqd W_k) 
\to 
{\tau_k}_{*} \mathcal{Q}_r(-(b+1)qd W_k) \right) 
\big) .
  \end{multline*}
 We are now reduced to showing that the last term equals $O(b^{n-2})$.

Let $D$ denote the effective Cartier divisor $qdW_k$.
Consider the exact sequence
$$
 \mathcal{Q}_r(-(b+1)D) \to
\mathcal{Q}_r(- bD) 
\to 
  \mathcal{Q}_{r}( -bD)\vert_D
  \to 
  0
$$
After pushing forward by $\tau_k$, we see
$$
{\tau_k}_{*} \mathcal{Q}_r(-(b+1)D) \to
{\tau_k}_{*} \mathcal{Q}_r(-bD) \to 
H^0(D,  \mathcal{Q}_{r}(-b D) \vert_D ) \to 0. 
$$
is exact for $b\gg0$. Since $\mathcal{Q}_r$ is supported on the locus of $Z_{0,F_k}$ where $\widetilde{G}$ is not $\Q$-Cartier, 
Lemma \ref{l-cod3Qfac} implies  $\mathcal{Q}_r \vert_{D}$ has dimension at most $\dim(Z_{0,F_k}) -3  = n-2$. 
Therefore,  
$$H^0(D,  \mathcal{Q}_{r}\vert_{D} (-bD\vert_{D} ) ) = O\left( b^{n-2} \right)$$
 and the lemma is complete. 
\end{proof}
The previous proof used the following property of lc pairs. 

\begin{lem}\label{l-cod3Qfac}
Let $(X,\Delta+E_1+E_2)$ be an lc pair such that (i) $(X,\Delta)$ is klt and (ii) $E_1$ and $E_2$ are $\Q$-Cartier prime divisors. 
If $x\in X$ is a codimension three point and $x\in E_1 \cap E_2 \cap \Supp(\Delta)$, then $X$ is $\Q$-factorial at $x$.
\end{lem}

\begin{proof}
After taking appropriate index one covers, we can assume $E_1$ and $E_2$ are Cartier. 
By cutting, we can assume $\dim(X)=3$ and $x$ is a closed point. 
Since $(X,\Delta)$ is klt,  $E_1$ is Cohen-Macaulay  \cite[5.25]{KM98}. 

We claim $E_1$ is normal at $x$. If not, since $E_1$ is $S_2$, it cannot be $R_1$ by Serre's Theorem. 
Hence, $E_1$ is singular on a curve $C$ passing through $x$. 
Note that  $C \not\subset \Supp (\Delta) \cup E_2$ by \cite[2.32]{Kol13}. 
If we consider the normalization $E_1^\nu\to E_1$, 
we see  $\Diff_{E_1^\nu}(\Delta)$ has coefficient one along the divisors in the preimage of $C$ and positive coefficient along the preimage of $\Supp(\Delta)$. 
By  \cite[2.31]{Kol13}, this implies that 
$(E_1^\nu,\Diff_{E_1^\nu}(\Delta) +E_2\vert_{E_1^\nu} )$ is not lc, which contradicts adjunction.

Shrinking around $x$, we may assume  $E_1$ is normal.
Adjunction gives ${(E_1,\Delta\vert_{E_1}+ E_2 \vert_{E_1} )}$ is lc. 
Since $E_{2} \vert_{E_1}$ is Cartier, $(E_1, \Delta\vert_{E_1})$ is canonical at $x$.  
Using that $x\in \Supp( \Delta\vert_{E_1})$, \cite[2.29.2]{Kol13} yields that $E_1$ is smooth at $x$. 
Hence, $X$ is smooth at $x$.
\end{proof}

\begin{proof}[Proof of Proposition \ref{p-initalcod2}]
To prove the statement for (1), it suffices to show 
\begin{equation}\label{e-Omn-2}
\sum_{p \geq 0 } \dim N_{p,m} = O(m^{n-2}),
\end{equation}
where $N_{p,m} := \init (I_{B_0})_{p,m} / I_{p,m}$. 
The previous estimate follows from Lemma \ref{l-idealgrowth}.

Indeed, by Lemma \ref{l-idealgrowth} and \eqref{e-Nideal}
\[
\sum_{m \geq 0 } \dim  N_{jd-mk,m}  = O \left( j^{n-2} \right) 
.\]
Since $ \cF$ is linearly bounded, there exists a positive integer $C$ so that $\gr_{\cF}^{p} R_m =0$ for all $p\geq mC$. 
Hence, $N_{p,m} =0$ for $p \geq mC$ and we see 
\[
\sum_{m=0}^{M} \sum_{p\geq 0}  \dim N_{pd,m} 
\leq 
\sum_{j=0}^{M(C+k)/d} \sum_{m\geq 0} \dim N_{jd-mk,m} 
= 
O\left( M^{n-1}\right)
\]
Therefore, 
\[
 \sum_{p\geq 0}  \dim N_{pd,m}  = O(m^{n-2})
. \]
Observe that $\gr_{\cF}^p R_m=0$ for all  $p$ not divisible by $d$ 
by Proposition \ref{p-F=F_E} and the fact that $E=d^{-1} F$. 
Therefore, the previous equation implies  \eqref{e-Omn-2} holds. Hence, (1) holds and (2) holds by an identical argument.
 \end{proof}

\begin{proof}[Proof of Corollary \ref{c-aut}]
The proof is the same as the proof of Corollary \ref{c-aut2}, 
but with Theorem \ref{t-sep} replaced by Theorem \ref{t-main} (3).
\end{proof}

\bigskip

\end{document}